\documentclass[11pt, a4paper,reqno]{amsart}

\usepackage[pagewise,mathlines]{lineno}
\usepackage{etoolbox}
\AtBeginEnvironment{multline*}{\linenomath}
\AtEndEnvironment{multline*}{\endlinenomath}
\AtBeginEnvironment{multline}{\linenomath}
\AtEndEnvironment{multline}{\endlinenomath}
\AtBeginEnvironment{align*}{\linenomath}
\AtEndEnvironment{align*}{\endlinenomath}
\AtBeginEnvironment{align}{\linenomath}
\AtEndEnvironment{align}{\endlinenomath}
\AtBeginEnvironment{aligned}{\linenomath}
\AtEndEnvironment{aligned}{\endlinenomath}
\AtBeginEnvironment{equation}{\linenomath}
\AtEndEnvironment{equation}{\endlinenomath}
\AtBeginEnvironment{equation*}{\linenomath}
\AtEndEnvironment{equation*}{\endlinenomath}
\AtBeginEnvironment{gather}{\linenomath}
\AtEndEnvironment{gather}{\endlinenomath}
\AtBeginEnvironment{gather*}{\linenomath}
\AtEndEnvironment{gather*}{\endlinenomath}
\AtBeginEnvironment{diagram}{\linenomath}
\AtEndEnvironment{diagram}{\endlinenomath}

\usepackage{enumerate}
\usepackage[PS]{diagrams}
\usepackage{xcolor}
\usepackage{pgf}

\newcommand{\udot}[1]{\ensuremath{\underset{\bullet}{#1}}}
\newcommand{\sst}{\scriptstyle}

\newcommand{\iso}{\cong}

\newcommand{\mcb}[2]{\colorbox{#1}{\(#2\)}}
\newcommand{\gmcb}[1]{\mcb{lime}{#1}}

\newcommand{\tot}{\ensuremath{\mathrm{Tot}}}
\newcommand{\sgn}{\ensuremath{\mathrm{sgn}}}
\newcommand{\inv}{^{-1}}

\newcommand{\totds}{\ensuremath{\mathrm{T}
    \raisebox {0.6pt} {\scalebox{0.6}{\ensuremath{+}}} \hskip 0.25 pt
    \llap{\rm o} \mathrm{t}\,}}
\newcommand{\tottr}{\ensuremath{{}_{\mathrm{tr}}\mathrm{Tot}\,}}

\newcommand{\prodtr}{\ensuremath{\sideset{_\mathrm{tr}}{}\prod}}

\newcommand{\torus}{\ensuremath{\mathcal{T}}}

\newcommand{\triv}{\mathfrak{Triv}}
\newcommand{\derived}{\mathfrak{Der}}

\newcommand{\ie}{\textit{i.e.}}
\newcommand{\resp}{\textit{resp.}}

\numberwithin{equation}{section}
\numberwithin{section}{part}   %  Use part.section numbering scheme

\newcommand{\ignore}[1]{\relax}

\newtheorem{theorem}[equation]{Theorem}
\newtheorem{corollary}[equation]{Corollary}
\newtheorem{proposition}[equation]{Proposition}
\newtheorem{lemma}[equation]{Lemma}

\theoremstyle{definition}
\newtheorem{definition}[equation]{Definition}

\newtheorem{example}[equation]{Example}

\newcommand{\bZ}{\mathbb{Z}}
\newcommand{\bR}{\mathbb{R}}
\newcommand{\bN}{\mathbb{N}}

\newcommand{\nbd}{\nobreakdash}
\newcommand{\id}{\ensuremath{\mathrm{id}}}
\newcommand{\tensor}{\otimes}
\newcommand{\LL}{{R[x,\,x\inv,\, y,\, y\inv]}}

\newcommand{\pL}{\ensuremath{\textrm{\rm \L}}}

\newcommand{\nov}[1]{(\!( {#1} )\!)} %Tighter ((-))
\newcommand{\powers}[1]{[\;\!\![ {#1} ]\;\!\!]} %Tighter [[-]]

\newcommand{\hook}[1]{\lfloor#1\rfloor}

\let\oldtocsubsection=\tocsubsection
\renewcommand{\tocsubsection}[2]{\hspace{3.5em}\(\cdot\)~\oldtocsubsection{#1}{#2}}

\usepackage[hypertexnames=false,colorlinks=false,pdfborderstyle={/S/U/W 1}]{hyperref}
\usepackage{bookmark}

\begin{document}

\title[Finite domination and Novikov rings]{Finite domination and Novikov rings.\\Laurent polynomial rings in several variables}

\date{\today}

\author{Thomas H\"uttemann}

\address{Thomas H\"uttemann\\ Queen's University Belfast\\ School of
  Mathematics and Physics\\ Pure Mathematics Research Centre\\ Belfast
  BT7~1NN\\ Northern Ireland, UK}

\email{t.huettemann@qub.ac.uk}

\urladdr{http://huettemann.zzl.org/}

\author{David Quinn}

\address{David Quinn\\ University of Aberdeen\\ School of Natural and
  Computing Sciences\\ Institute of Mathematics\\ Fraser Noble Building\\
  Aberdeen AB24~3UE\\ Scotland, UK}

\email{davidquinnmath@gmail.com}

\thanks{This work was supported by the Engineering and Physical
  Sciences Research Council [grant number EP/H018743/1].}



\begin{abstract}
  % At a recent seminar, the following question was asked, arising from
  % the definition of a homotopy commutative square --- ``What is a
  % homotopy commutative cube?'' This paper gives one answer. Our main
  % example is a non-trivial homotopy commutative cube constructed
  % explicitly from a commutative cubical diagram.
  %
  % At a subsequent seminar, the following question was asked, arising
  % from the definition of a homotopy commutative cube --- ``What is a
  % homotopy commutative cube good for?'' This paper gives one
  % answer. Our main result is a necessary criterion for finite
  % domination of chain complexes of modules over a Laurent polynomial
  % ring in several indeterminates in terms of their \textsc{Novikov}
  % cohomology.
  We present a homological characterisation of those chain complexes
  of modules over a \textsc{Laurent} polynomial ring in several
  indeterminates which are finitely dominated over the ground ring
  (that is, are a retract up to homotopy of a bounded complex of
  finitely generated free modules). The main tools, which we develop
  in the paper, are a non-standard totalisation construction for
  multi-complexes based on truncated products, and a high-dimensional
  mapping torus construction employing a theory of cubical diagrams
  that commute up to specified coherent homotopies.
\end{abstract}

\maketitle

\tableofcontents

\part*{Introduction}

Let $R \subseteq K$ be a pair of unital rings. A cochain complex~$C$
of $K$\nbd-modules is called {\it $R$-finitely dominated\/} if $C$~is
homotopy equivalent, as an $R$\nbd-module complex, to a bounded
complex of finitely generated projective $R$\nbd-modules;
equivalently, if $C$ is a retract up to homotopy of a bounded complex
of finitely generated free $R$\nbd-modules
\cite[Proposition~3.2]{MR815431}.

Finite domination is relevant, for example, in group theory and
topology. Suppose that $G$ is a group of type~$(FP)$; this means, by
definition, that the trivial $G$\nbd-module~$\bZ$ admits a finite
resolution~$C$ by finitely generated projective
$\bZ[G]$\nbd-modules. Let $H$~be a subgroup of~$G$. Deciding whether
$H$ is of type~$(FP)$ is equivalent to deciding whether $C$ is
$\bZ[H]$\nbd-finitely dominated.

In topology, finite domination has been considered in the context of
homological finiteness properties of covering spaces (\textsc{Dwyer}
and \textsc{Fried} \cite{MR887774}), or properties of ends of
manifolds (\textsc{Ranicki} \cite{Ranicki-findom}).

\medbreak

Our starting point is the following result of \textsc{Ranicki}
\cite[Theorem~2]{Ranicki-findom}: {\it Let $C$ be a bounded complex
of finitely generated free modules over $K = R[x,x\inv]$. The complex
$C$ is $R$\nbd-finitely dominated if and only if the two complexes
  \[C \tensor_K R\nov{x} \quad \text{and} \quad C \tensor_K
  R\nov{x\inv}\] are acyclic.}  Here $R\nov{x} = R\powers{x}[x\inv]$
denotes the ring of formal \textsc{Laurent} series in~$x$, and
$R\nov{x\inv} = R\powers{x\inv}[x]$ denotes the ring of formal
\textsc{Laurent} series in~$x\inv$.

For \textsc{Laurent} polynomial rings in several indeterminates, it is
possible to strengthen this result to allow for iterative application,
see for example~\cite{Iteration}. In particular, writing $L = \LL$ for
the \textsc{Laurent} polynomial ring in two variables, one can show
that a bounded complex of finitely generated free $L$\nbd-modules is
$R$\nbd-finitely dominated if and only if the four complexes
% linenocrutch
\begin{align*}
  &C \tensor_L R[x,x\inv]\nov{y}  && C \tensor_L R[x,x\inv]\nov{y\inv} \\
  &C \tensor_{R[x,x\inv]} R\nov{x} && C \tensor_{R[x,x\inv]} R\nov{x\inv}
\end{align*}
are acyclic.

This characterisation has the rather unsatisfactory feature that the
tensor products are taken over two different rings. In the present
paper, we propose a different, non-iterative approach leading to an
entirely new characterisation of finite domination. Roughly speaking,
a cone in~$\bR^{n}$ determines a certain ring of formal
\textsc{Laurent} series, and our main theorem asserts that finite
domination is equivalent to the vanishing of homology with
coefficients in these \textsc{Laurent} series rings for sufficiently
many cones.

\section*{Informal statement of results}

We think of the \textsc{Laurent} polynomial ring $L =
R[x_{1}^{\pm1},\, x_{2}^{\pm1},\, \cdots,\, x_{n}^{\pm1}]$ in
$n$~indeterminates as the monoid ring $R[\bZ^{n} \cap \bR^{n}]$,
identifying the $n$\nbd-tuples of exponents of monomials with integral
points in~$\bR^{n}$. A cone $\sigma \subseteq \bR^{n}$ then defines a
subring $R[\bZ^{n} \cap \sigma^{\vee}]$ of~$L$, where $\sigma^{\vee}$
is the dual cone. Another way to describe this is as follows. The
support of a formal sum $f = \sum_{\mathbf{a} \in \bZ^{n}}
r_{\mathbf{a}}\mathbf{x}^{\mathbf{a}}$, where $\mathbf{x}^{\mathbf{a}}
= x_{1}^{a_{1}}x_{2}^{a_{2}} \cdots x_{n}^{a_{n}}$, is the set of
those $\mathbf{a} \in \bZ^{n}$ with $r_{\mathbf{a}} \neq 0$. Then
$R[\bZ^{n} \cap \sigma^{\vee}]$ is the set of all such $f$ having
finite support which is contained in $\sigma^{\vee}$. Omitting the
finiteness condition yields a set of formal \textsc{Laurent} series;
if $\dim \sigma = n$ the ring $R\nov{\sigma}$ can be thought of as the
set of those $f$ with support in a translated copy of~$\sigma^{\vee}$,
where we allow translation by the negative of a vector in the interior
of~$\sigma^{\vee}$. This construction is modified for cones of
dimension less than~$n$; the modification, and the ring structure
of~$R\nov{\sigma^{\vee}}$, are explained in detail in
\S\ref{sec:cones}.

For example, if $\sigma$ is a cone spanned by $d$ elements
of~$\bZ^{n}$ which are $\bZ$\nbd-linearly independent, then there is
an $R$\nbd-algebra isomorphism
\begin{displaymath}
  R \nov{\sigma} \iso R[x_{1}^{\pm1},\, x_{2}^{\pm1},\, \cdots,\,
  x_{n-d}^{\pm1}] \powers{x_{n-d+1},\, x_{n-d+2},\, \cdots,\, x_{n}}
  [1/\prod_{k=0}^{d-1} x_{n-k}]
\end{displaymath}
between $R \nov{\sigma}$ and a localisation of a formal power series
ring over a \textsc{Laurent} polynomial ring; the right-hand side
consists of formal \textsc{Laurent} series (with coefficients in a
ring of \textsc{Laurent} polynomials) having the property that the
support lies in the ``first orthant'' after shifting along the
diagonal vector $(1,\, 1,\, \cdots,\, 1)$ finitely often.

In Theorems~\ref{thm:findom_triv_1st_quadrant}
and~\ref{thm:toric_criterion} we show that if $C$ is a bounded complex
of finitely generated free $L$\nbd-modules which is $R$\nbd-finitely
dominated then the complex $C \tensor_{L} R\nov{\sigma^{\vee}}$ is
acyclic. We also show that, conversely, if $C \tensor_{L}
R\nov{\sigma^{\vee}}$ is acyclic for all $\sigma$ coming from a family
of cones covering all of~$\bR^{n}$, then $C$ is necessarily
$R$\nbd-finitely dominated.

For two variables ($n=2$) a version of this programme has been carried
out by the authors in the paper~\cite{square}. The present extension
to more than two variables is non-trivial as it demands a theory of
high-dimensional mapping tori, in turn resting on a theory of homotopy
commutative cubical diagrams. Both are developed in this paper, and
might be of independent interest for researchers in homological
algebra.

\section*{Structure of the paper}

The paper is divided into three parts. In the first we develop the
theory of homotopy commutative cubical diagrams of cochain complexes,
culminating in the construction of derived cubes
(Theorem~\ref{thm:main_cube}) which are, roughly speaking, homotopy
commutative diagrams obtained from commutative ones by replacing all
entries with homotopy equivalent ones. In the second part, we introduce
higher-dimensional mapping tori, and prove an algebraic
high-dimensional analogue of \textsc{Mather}'s trick of turning a
mapping torus of a composite map through an angle of~$\pi$
(Lemmas~\ref{lem:mather_for_analogue} and~\ref{lem:mather}). In the
third part, we define a non-standard totalisation construction for
multi-complexes based on truncated products rather than direct sums or
products, and prove a simple vanishing criterion
(Proposition~\ref{prop:tr_tot_acyclic}). The main result is then
proved by a combination of multi-complex and mapping torus techniques.

\section*{Conventions}

We fix some notation to be used throughout the paper. Let $N$~be a
totally ordered indexing set with $n$ elements; we will mostly work
with the set $N = \{1,\, 2,\, \cdots,\, n\}$. The letters $A$, $B$
and~$S$ will denote subsets of~$N$ with cardinalities $a$, $b$
and~$s$, respectively, unless explicitly defined otherwise. Let $R$ be
an arbitrary unital ring. Modules will always be right modules if not
specified otherwise. Our complexes will be indexed cohomologically
(differentials increase the degree), and will thus be termed cochain
complexes. A cochain complex~$C$ is bounded above (\resp, bounded) if
$C^{n} = 0$ for all $n \gg 0$ (\resp, $|n| \gg 0$).  The cohomology
modules of an arbitrary cochain complex $C$ with differential~$d$ are
defined as usual as the quotient modules
\[H^{n}(C) = \big(\ker (d \colon C^{n} \rTo C^{n+1})\big) \big/ \big(
\mathrm{Im} (d \colon C^{n-1} \rTo C^{n})\big) \ .\] A map of cochain
complexes is called a quasi-isomorphism if it induces isomorphisms on
all cohomology modules. A standard result in homological algebra
asserts that a quasi-isomorphism between bounded-above cochain
complexes of projective modules is a homotopy equivalence. We will use
this result frequently throughout the paper.

\part{Homotopy commutative cubes}
\label{part:cubes}

A commutative square diagram of cochain complexes and cochain maps,
\begin{diagram}
  C_2 & \rTo^{f_{12,2}} & C_{12} \\
  \uTo<{f_{2,\emptyset}} && \uTo>{f_{12,1}} \\
  C_\emptyset & \rTo^{f_{1, \emptyset}} & C_{1}
\end{diagram}
can be considered as a three-fold cochain complex with commuting
differentials; we think of the cochain direction as the
last-coordinate direction, or $z$\nbd-direction, with the square
sitting at coordinates $x=0,1$ and $y=0,1$. The totalisation $T$ then
is a cochain complex given by
\[T^n = C_\emptyset^n \oplus \big( C_1^{n-1} \oplus C_2^{n-1} \big)
\oplus C_{12}^{n-2} \ ,\] with differential described by the following
matrix (suppressing zero entries):
\[D =
\begin{pmatrix}
  d_{\emptyset} \\
  f_{1, \emptyset} & -d_{1} \\
  f_{2, \emptyset} && -d_{2} \\
  & -f_{12,1} & f_{12,2} & d_{12}
\end{pmatrix}\] The construction can be extended to higher-dimensional
cubes in a standard manner. It is a different matter altogether what
happens for diagrams that commute up to homotopy only. For the square
above, suppose that $H$ is a homotopy between $f_{12,1} \circ f_{1,
  \emptyset}$ and~$f_{12,2} \circ f_{2, \emptyset}$. The above
totalisation construction, if applied verbatim, fails to result in a
cochain complex. However, if the matrix $D$ is modified to include the
specified homotopy
\[\begin{pmatrix}
  d_{\emptyset} \\
  f_{1, \emptyset} & -d_{1} \\
  f_{2, \emptyset} && -d_{2} \\
  \pm H & -f_{12,1} & f_{12,2} & d_{12}
\end{pmatrix}\] (the sign depending on the direction of the homotopy),
then we obtain a cochain complex again, and it seems justified to
consider this as the ``right'' totalisation construction for homotopy
commutative squares equipped with a choice of homotopy.

In this first part of the paper, we {\it define\/} a homotopy
commutative cube to be a collection of cochain maps, homotopies and
higher homotopies with the characteristic property that a totalisation
construction, similar to the one above and to be detailed below,
results in a cochain complex. In fact, for reasons of aesthetics we go
one step further: the data we consider consists of a collection of
graded modules, together with module maps $H_{B,A}$, indexed by pairs
$A \subseteq B$ of subsets of a given indexing set~$N$, with $H_{B,A}$
being of degree $1-\#A-\#B$. The main advantage is that the
differentials (which correspond to the case $A=B$) are now treated in
exactly the same way as all the other structure maps, leading to a
slightly more symmetric definition of homotopy commutative cubes. It
is then an easy exercise, solved in~\S\ref{sec:entries} below, to show
that the data defining a homotopy commutative cube consists of cochain
complexes, cochain maps and, for each two-dimensional face of the
cube, homotopies between the cochain maps. Data associated to
higher-dimensional faces should then be interpreted as higher
homotopies, or coherence data.

We then specialise to homotopy commutative cubes which have the same
cochain complex attached to each vertex, and in which the structure
maps depend on the direction in the cube only (and not on their
position within the cube). Here the main point is to realise that the
definition of homotopy commutativity leads to a consistent, meaningful
notion. This is recorded as Lemma~\ref{lem:consistency}.

Finally, we include a non-trivial example of a homotopy commutative
cube which is constructed from a commutative cubical diagram. Details
are contained in Theorem~\ref{thm:main_cube} and its proof. This
construction of ``derived'' homotopy-commutative cubical diagrams will
be an essential ingredient for the analysis of higher-dimensional
mapping tori in later parts of the paper.

\section{Total incidence numbers}
\label{sec:incidence-numbers}

\begin{definition}
  \label{def:iterated_incidence}
  Let $B = \{b_1 < b_2 < \ldots < b_b\} \subseteq N$ and write $d_i(B)
  = B \setminus \{b_i\}$, for $1 \leq i \leq b$. Given a subset $A
  \subset B$ there is a unique way to write $A = d_{i_1} d_{i_2}
  \cdots d_{i_{b-a}} (B)$ with $i_1 < i_2 < \ldots < i_{b-a}$, and we
  define
  \[[B:A] = (-1)^{b-a} (-1)^{i_1 + i_2 + \ldots + i_{b-a}} \ .\] For $A
  \not\subseteq B$ we define $[B:A] = 0$, and we set $[A:A] = 1$ for
  any~$A$. In either case we call $[B:A]$ the {\it total incidence
    number\/} of~$A$ and~$B$.
\end{definition}

From the simplicial identity $d_i \circ d_j (B)= d_{j-1} \circ d_i
(B)$ for $i<j$ we immediately infer that
\begin{multline}
  \label{eq:simplicial_identities}
  \big[A \amalg \{x,y\} : A \amalg \{x\}\big] \cdot \big[A \amalg
  \{x\} : A\big] \\ = - \big[A \amalg \{x,y\} : A \amalg \{y\}\big]
  \cdot \big[A \amalg \{y\} : A\big]
\end{multline}
for distinct elements $x,y \in N \setminus A$. We also have the
following combinatorial re-statement of the definition of total
incidence numbers:

\begin{lemma}
  \label{lem:interpretation}
  For $B \supseteq A$ we have $[B:A] = (-1)^\kappa$ where $\kappa$~is
  the number of pairs $(b,x) \in B \times (B \setminus A)$ with
  $b<x$. In particular,
  \[[B:\emptyset] = (-1)^{b(b-1)/2} = \begin{cases} 1 & \text{if } b
    \equiv 0,1 \mod 4 \\ -1 & \text{if } b \equiv 2,3 \mod
    4\end{cases}\] and $[B:B \setminus \{z\}] = (-1)^{\#\{b \in B
    \,|\, b < z\}}$ for $z \in B$.\qed
\end{lemma}

\begin{lemma}
  \label{lem:incidence_independent}
  Given sets $B \supseteq S \supseteq A$ and an element $z \in N
  \setminus B$, the products of total incidence numbers\footnote{We
    simplify notation and write $A \amalg z$ instead of the more
    precise $A \amalg \{z\}$.}
  \[[B:S][S:A] \quad \text{and} \quad [B \amalg z:S \amalg z][S \amalg
  z:A \amalg z]\] differ by a factor $\epsilon = \pm 1$ that is
  independent of~$S$ (that is, depends only on $B$, $A$ and~$z$). We
  also have $[B \amalg z:A \amalg z] = \epsilon \cdot [B:A]$, with the
  same factor~$\epsilon$.
\end{lemma}

\begin{proof}
  Write $S = d_{i_1} d_{i_2} \cdots d_{i_{b-s}} (B)$ with $i_1 < i_2 <
  \ldots < i_{b-s}$, and $B = S \amalg \{z_1 < z_2 < \ldots <
  z_{b-s}\}$. Then clearly $S \amalg z = d_{j_1} d_{j_2} \cdots
  d_{j_{b-s}} (B \amalg z)$ where
  \[j_\ell = \begin{cases} i_\ell & \text{if } z_\ell < z \ , \\
    i_\ell + 1 & \text{if } z_\ell > z \ . \end{cases}\] Consequently,
  $[B \amalg z:S \amalg z] = (-1)^\kappa [B:S]$ where $\kappa$ is the
  number of elements in~$B \setminus S$ which are bigger than~$z$.

  \smallskip

  Now, re-defining some of the symbols above, write $A = d_{i_1}
  d_{i_2} \cdots d_{i_{s-a}} (S)$ with $i_1 < i_2 < \ldots < i_{s-a}$,
  and $S = A \amalg \{z_1 < z_2 < \ldots < z_{s-a}\}$. Then clearly $A
  \amalg z = d_{j_1} d_{j_2} \cdots d_{j_{s-a}} (S \amalg z)$ where
  \[j_\ell = \begin{cases} i_\ell & \text{if } z_\ell < z \ , \\
    i_\ell + 1 & \text{if } z_\ell > z \ . \end{cases}\] Consequently,
  $[S \amalg z:A \amalg z] = (-1)^\nu [S:A]$ where $\nu$ is the number
  of elements in~$S \setminus A$ which are bigger than~$z$.

  \smallskip

  In total, the two products of incidence numbers differ thus by a
  factor $(-1)^{\kappa + \nu}$. But $\kappa + \nu$ is the number of
  elements in $(B \setminus S) \amalg (S \setminus A) = B \setminus A$
  which are larger than~$z$. This number does not depend on~$S$.

  \smallskip
  
  The last assertion is the special case $S = B$ as $[B:A] =
  [B:B][B:A]$ and similarly $[B \amalg z:A \amalg z] = [B \amalg z:B
  \amalg z] [B \amalg z:A \amalg z]$.
\end{proof}

\section{$N$-diagrams and their totalisation}
\label{sec:leitch}

\begin{definition}
  \label{def:pre-leitch}
  An {\it $N$\nbd-diagram\/} consists of the following
  (non-functorial) data:
  \begin{itemize}
  \item for each $A \subseteq N$ a graded $R$\nbd-module $F(A) =
    \bigoplus_{k \in \bZ} F(A)^{k}$;
  \item for each inclusion $A \subseteq B$ a graded module
    homomorphism \[H_{B,A} \colon F(A) \rTo F(B)\] of degree $a-b+1$.
  \end{itemize}
  We write $d = d_A$ in place of~$H_{A,A}$, and for $b-a=1$ we denote
  $H_{B,A}$ by $f_{B,A}$. (The data defining an $N$\nbd-diagram is
  not required to satisfy any compatibility conditions.)
\end{definition}

\begin{definition}
  \label{def:totalisation}
  Let $F$ be an $N$\nbd-diagram. The {\it totalisation of~$F$}
  consists of the graded $R$\nbd-module $\tot(F)$ given by
  \[\tot(F)^\ell = \bigoplus_{A \subseteq N} F(A)^{\ell-a} \ ,\] and module
  homomorphisms \[D(F) = D = (D_{B,A})_{A,B \subseteq N} \colon
  \tot(F)^\ell \rTo \tot(F)^{\ell + 1}\] given by \[D_{B,A}
  = \begin{cases} (-1)^{ab} [B:A] \cdot H_{B,A} & \text{if } A
    \subseteq B \ , \\ 0 & \text{otherwise.} \end{cases}\]
\end{definition}

We can, and will, think of~$D(F)$ as a matrix with columns and rows
indexed by the subsets of~$N$; note that $D_{B,A} \colon F(A)^{\ell-a}
\rTo F(B)^{(\ell+1)-b}$ is a map of degree~$a-b+1$. The
composition \[D(F) \circ D(F) \colon \tot(F)^\ell \rTo
\tot(F)^{\ell+2}\] is calculated by matrix multiplication;
explicitly\footnote{the variable over which the sum is taken is
  sometimes marked by a dot underneath},
\begin{equation}
  \label{eq:composition}
  \big( D(F) \circ D(F) \big)_{B,A} = \sum_{A \subseteq \udot S \subseteq B}
  (-1)^{bs} (-1)^{sa} [B:S][S:A] \cdot H_{B,S} \circ H_{S,A} \ .
\end{equation}

\section{Homotopy commutative $N$-cubes}
\label{sec:entries}

\begin{definition}
  \label{def:homotopy_commutative}
  A {\it homotopy commutative $N$\nbd-cube\/} is an $N$\nbd-diagram
  $F$ such that its totalisation is a cochain complex of
  $R$\nbd-modules with differential~$D(F)$ (that is, such that $D(F)
  \circ D(F) = 0$).
\end{definition}

For the remainder of this section we consider a homotopy commutative
$N$\nbd-cube~$F$ with totalisation $T = \tot(F)$ and associated
differential $D=D(F)$.

\begin{lemma}
  The graded module~$F(A)$ is a cochain complex of $R$\nbd-modules
  with differential~$d_A$.
\end{lemma}

\begin{proof}
  By definition of a homotopy commutative $N$\nbd-cube we have $D
  \circ D = 0$, and so in particular, using~\eqref{eq:composition},
  $d_A \circ d_A = (D \circ D)_{(A,A)} = 0$.
\end{proof}

\begin{lemma}
  For $b-a = 1$ the map $f_{B,A} = H_{B,A} \colon F(A) \rTo F(B)$ is a
  cochain map.
\end{lemma}

\begin{proof}
  The $(B,A)$\nbd-entry of $D \circ D$ is given, according
  to~\eqref{eq:composition}, by
  \[(-1)^{b^2} (-1)^{ba} [B:A] \cdot d_B \circ f_{B,A} + (-1)^{ba}
  (-1)^{a^2} [B:A] \cdot f_{B,A} \circ d_A \ .\] Since $D \circ D =
  0$, and since $(-1)^{b^2} = - (-1)^{a^2}$ this implies that
  $f_{B,A}$ is a cochain map as claimed.
\end{proof}

\begin{lemma}
  Let $B = A \amalg \{z_0 < z_1\}$, and write $Z_i = A \amalg \{z_i\}
  = B \setminus \{z_{1-i}\}$. Then $H_{B,A}$ is a homotopy from
  $f_{B,Z_0} \circ f_{Z_0, A}$ to $f_{B,Z_1} \circ f_{Z_1, A}$ so that
  \[d_B \circ H_{B,A} + H_{B,A} \circ d_A = f_{B,Z_1} \circ f_{Z_1,A}
  - f_{B,Z_0} \circ f_{Z_0, A} \ .\]
\end{lemma}

\begin{proof}
  Again we will use that the $(B,A)$\nbd-entry of $D \circ D$ has to
  be trivial. That is, we must have
  \begin{multline*}
    (-1)^{b} (-1)^{ba} [B:A] \cdot d_B \circ H_{B,A} \\
    + (-1)^{b(b-1)} (-1)^{(a+1)a} [B:Z_0] [Z_0:A] \cdot f_{B,Z_0}
    \circ f_{Z_0,A} \\
    + (-1)^{b(b-1)} (-1)^{(a+1)a} [B:Z_1] [Z_1:A] \cdot f_{B,Z_1}
    \circ f_{Z_1,A} \\
    + (-1)^{ba} (-1)^{a} [B:A] \cdot H_{B,A} \circ d_A = 0 \ .
  \end{multline*}
  (We have used $(-1)^{k^2} = (-1)^k$ and $[S:S] = 1$.) Now $(-1)^{b}
  (-1)^{ba} = (-1)^{ba} (-1)^{a} = 1$ as $a$ and~$b$ have the same
  parity, and we similarly have $(-1)^{b(b-1)} (-1)^{(a+1)a} = 1$.
  
  \smallskip

  Write $A = d_i d_j (B)$ with $i<j$. Then $[B:Z_0] = (-1)^j$ and
  $[Z_0:A] = (-1)^i$ so that $[B:Z_0][Z_0:A] = (-1)^{j+i} = [B:A]$.

  \smallskip

  But we also have $A = d_{j-1} d_i (B)$; this implies $[B:Z_1] =
  (-1)^i$ and $[Z_1:A] = (-1)^{j-1}$ so that $[B:Z_1][Z_1:A] =
  (-1)^{j+i-1} = -[B,A]$.

  \smallskip

  Cancelling the common factor of~$[B:A]$ and re-arranging the terms
  now gives
  \[d_B \circ H_{B,A} + H_{B,A} \circ d_A = f_{B,Z_1} \circ f_{Z_1,A}
  - f_{B,Z_0} \circ f_{Z_0, A} \ ,\] proving the Lemma.
\end{proof}

In this vein one can work out conditions on the maps $H_{B,A}$ for
$b-a \geq 3$; these maps provide what are sometimes called higher
(coherent) homotopies.

\section{Special $N$-diagrams and special $N$-cubes}
\label{sec:special-N-cubes}

\begin{definition}
  \label{def:special_diagrams_and_cubes}
  An $N$\nbd-diagram is called {\it special\/} if the graded
  $R$\nbd-module $F(A)$ does not depend on~$A$ (that is, if $F(A) =
  F(B)$ for all $A,B \subseteq N$), and if the maps $H_{B,A}$ depend
  only on $B \setminus A$. We will usually use the following notation
  for a special $N$-diagram:
  \begin{itemize}
  \item $C$ is the graded $R$\nbd-module $F(A)$, for any $A \subseteq
    N$;
  \item $d = d_A$, for any $A \subseteq N$;
  \item $f_k = f_{\{k\},\emptyset}$, for $k \in N$;
  \item $H_S = H_{B,A}$, for any $A \subseteq B$ with $B \setminus A =
    S$, $s \geq 2$.
  \end{itemize}
  Such data $C$, $d$, $f_k$ and $H_S$ determine, conversely, a
  special $N$\nbd-diagram by setting $d_A = d$, $H_{B,A} = H_{B
    \setminus A}$ if $b-a \geq 2$, and $f_{B,A} = f_{B \setminus A}$
  for $b-a = 1$.

  A {\it special $N$\nbd-cube\/} is a special $N$\nbd-diagram which is
  also a homotopy com\-muta\-tive $N$\nbd-cube.
\end{definition}

As shown in \S\ref{sec:entries} a special $N$\nbd-cube has the
property that $C$~is a cochain complex with differential~$d$, that the
maps $f_k \colon C \rTo C$ are cochain maps, and that $H_{\{k<\ell\}}$
is a homotopy from $f_\ell \circ f_k$ to $f_k \circ f_\ell$. This was
deduced from the equation $D(F) \circ D(F) = 0$ which is the defining
property of an $N$\nbd-cube.

\medbreak

Let us now analyse the data specifying higher homotopies in more
detail. Let $T$ denote the totalisation of a special $N$\nbd-diagram,
with associated map $D = D(F)$. Let $B \supseteq A$ and $C$ be given
with $C \cap B = \emptyset$ and $b-a \geq 3$. Then considering the
entries $(B,A)$ and $(B \amalg C, A \amalg C)$ of the equation $D
\circ D = 0$ we obtain two different conditions involving the map
$H_{B \setminus A}$, and they should be consistent in order to make
the definition of a special $N$\nbd-cube meaningful. It is clearly
enough to do so for $C$~a one element set.

\begin{lemma}
  \label{lem:consistency}
  Given sets $B \supseteq A$ and an element $z \in N \setminus B$, the
  two maps
  \[(D \circ D)_{B,A} \quad \text{and} \quad (D \circ D)_{B \amalg z,
    A \amalg z}\] agree up to sign.
\end{lemma}

\begin{proof}
  For $B \supseteq S \supseteq A$ the summands
  in~\eqref{eq:composition} corresponding to $S$ and~$S \amalg z$,
  respectively, agree up to sign since our homotopy commutative
  $N$\nbd-cube is a special $N$\nbd-cube. So it remains to check that
  the difference of signs does not depend on~$S$. In view of the
  previous Lemma it is enough to check that the difference in parity
  of $(b+1)(s+1)+(s+1)(a+1) = (s+1)(b+a+2)$ and $bs+sa = s(b+a)$ does
  not depend on~$s$. But if $b+a$ is even both numbers are even, while
  if $a+b$ is odd the two numbers have different parity, independent
  of~$S$.
\end{proof}

\begin{corollary}
  \label{cor:criterion}
  A cochain complex $C$ with differential~$d$, cochain maps $f_k
  \colon C \rTo C$ for $k \in N$, and maps $H_S \colon C \rTo C$ of
  graded modules of degree $1-s$, for $s \geq 2$, determine a special
  $N$\nbd-cube if and only if for all $S = \{k < \ell\}$ the
  map~$H_{\{k<\ell\}}$ is a homotopy from $f_\ell \circ f_k$ to~$f_k
  \circ f_\ell$, and for every subset $S \subseteq N$, $s \geq 3$, we
  have
  \begin{multline*}
    0 = (-1)^s [S:\emptyset] \cdot d \circ H_S + [S:\emptyset] \cdot
    H_S \circ d \\
    + \sum_{z \in S} [S:S \setminus z][S \setminus z:\emptyset] \cdot
    f_z \circ H_{S \setminus z} \\
    + (-1)^s \sum_{z \in S} [S:z] \cdot H_{S \setminus z} \circ f_z \\
    + \sum_{\substack {\udot T \subseteq S \\ t \geq 2 \leq s-t}}
    (-1)^{ts} [S:T][T:\emptyset] \cdot H_{S \setminus T} \circ H_T \ ,
  \end{multline*}
  where $t = \# T$ in the last sum (so that $T$~varies over all
  subsets of~$S$ such that both $T$ and~$S \setminus T$ have at least
  two elements).
\end{corollary}

\begin{proof}
  Let $D$ be the map associated to the totalisation of the special
  $N$\nbd-dia\-gram determined by the given data. This diagram is a
  special $N$\nbd-cube if and only if $D \circ D = 0$, which happens
  if and only if $(D \circ D)_{B,A} = 0$ for every pair of sets $B
  \supseteq A$. In view of Lemma~\ref{lem:consistency} this is
  equivalent to the condition $(D \circ D)_{S, \emptyset} = 0$ for
  every $S \subseteq N$. Now for $s=0$ this in turn means that $d$ is
  a differential, for $s=1$ this is equivalent to the $f_k$~being
  cochain maps, and for $s \geq 2$ this is a reformulation of the
  hypotheses on the maps~$H_S$ (bearing formula~\eqref{eq:composition}
  in mind).
\end{proof}

\section{Main examples}
\label{sec:explicit}

\subsection*{Commutative diagrams} A cochain complex~$C$
of $R$\nbd-modules with differential~$d$ together with a collection of
pairwise commuting cochain maps $f_k \colon C \rTo C$ for $k \in N$
defines a special $N$\nbd-cube upon setting $H_S = 0$ for $s \geq
2$. That is, a commutative cubical diagram of self-maps of~$C$ can be
considered in an obvious way as a homotopy commutative
$N$\nbd-cube. We will refer to such a special $N$\nbd-cube as a {\it
  trivial\/} one, and use the symbol $\triv(C; f_1,\, f_2,\, \cdots,\,
f_n)$ to denote the corresponding $N$\nbd-diagram. The totalisation of
a trivial cube is in fact the totalisation of a multi-complex, see
Proposition~\ref{prop:trivial_diagram_is_totalisation} below.

\subsection*{Mapping cones} For $N = \{1\}$ a special
$N$\nbd-cube is specified by a cochain map $f_1 \colon C \rTo C$. Its
totalisation has associated differential
\[D =
\begin{pmatrix}
  d & 0 \\ f_1 & -d
\end{pmatrix} \ ,\] and is one version of the mapping cone of~$f_1$,
up to shift. (This differs from other versions in sign and indexing
conventions.)

\subsection*{Square diagrams} Slightly more interesting
is the case of $N = \{1<2\}$, two self-maps $f_1, f_2 \colon C \rTo C$
and a specified homotopy $H = H_N \colon f_2 \circ f_1 \simeq f_1
\circ f_2$ so that $d \circ H_N - H_N \circ d = f_1 \circ f_2 - f_2
\circ f_1$. This data describes a special $N$\nbd-diagram; the map $D$
associated to its totalisation takes the form (omitting trivial
entries, but showing indexing subsets of~$N$)
\[D = \bordermatrix{ &\sst \emptyset & \sst \{1\} & \sst \{2\} & \sst
  N \cr \sst \emptyset & d \cr \sst \{1\} & f_1 & -d \cr \sst \{2\} &
  f_2 && -d \cr \sst N & -H & -f_2 & f_1 & d}\] which is easily shown
to satisfy $D \circ D = 0$ by direct calculation. That is, the given
data does in fact specify a special $\{1<2\}$\nbd-cube. Up to shift,
sign and naming conventions, the totalisation of a $\{1<2\}$\nbd-cube
is the ``mapping $2$\nbd-torus analogue'' discussed by the authors in
\cite[\S{}III.7]{square}.

\subsection*{Derived homotopy commutative cubes} Our most complicated
example shows how one can construct (special) homotopy commutative
diagrams from strictly commutative ones, roughly speaking by taking a
commutative cubical diagram and replacing all occurring cochain
complexes by homotopy equivalent ones. The construction is a main
ingredient for manipulations of mapping tori later in the paper.

\begin{theorem}
  \label{thm:main_cube}
  Let $D$ be a cochain complex of $R$\nbd-modules with
  differential~$d$. Let $h_k \colon D \rTo D$, $k \in N$, be a
  collection of pairwise commuting cochain maps. Let $C$ be another
  cochain complex of $R$\nbd-modules, with differential denoted by~$d$
  as well, and let $\alpha \colon C \rTo D$ and $\beta \colon D \rTo
  C$ be cochain maps. Suppose that $G$ is a homotopy from $\id_D$
  to~$\alpha \circ \beta$ so that $d \circ G + G \circ d = \alpha
  \circ \beta - \id_D$. Then $C$, $d$ and the following data define a
  special $N$\nbd-cube\footnote{leaving out the composition symbol
    $\circ$ occasionally, as is standard}:
  \begin{itemize}
  \item $f_k = \beta \circ h_k \circ \alpha$, for $k \in N$;
  \item $H_{\{k < \ell\}} = \beta \circ ( h_k G h_\ell - h_\ell G h_k)
    \circ \alpha$, for $\{k < \ell\} \subseteq N$;
  \item $H_S = \beta \circ \sum_{\sigma \in \Sigma(S)} \sgn(\sigma)
    h_{\sigma(z_1)} G h_{\sigma(z_2)} G h_{\sigma(z_3)} \ldots G
    h_{\sigma(z_s)} \circ \alpha$, for every set $S = \{z_1 < z_2 <
    \ldots < z_s\} \subseteq N$ with $s \geq 3$.
  \end{itemize}
\end{theorem}

\begin{definition}
  \label{def:derived_cube}
  The special $N$\nbd-cube defined in the previous Theorem is denoted
  $\derived(C;\, \alpha,\, \beta,\, G;\, h_1,\, h_2,\, \cdots,\,
  h_n)$, and is called the special $N$\nbd-cube {\it derived\/} from
  the trivial $N$\nbd-cube $\triv(D;\, h_1,\, h_2,\, \cdots,\, h_n)$.
\end{definition}

\begin{proof}[Proof of Theorem~\ref{thm:main_cube}]
  We verify that the hypotheses of Corollary~\ref{cor:criterion} are
  satisfied. First note that since $G$ has degree~$-1$, the maps~$H_S$
  have degree $1-s$ as required ($s \geq 2$). For $S = \{k < \ell\}$
  we have $s=2$ and, since $\alpha$, $\beta$ and all the~$h_j$ are
  cochain maps,
  \begin{align*}
    d \circ H_S + H_S \circ d & = d \circ \beta ( h_k G h_\ell -
    h_\ell G h_k)\alpha + \beta ( h_k G h_\ell - h_\ell G h_k)\alpha
    \circ d \\
    &= \beta ( h_k d G h_\ell - h_\ell d G h_k)\alpha + \beta ( h_k G
    d h_\ell - h_\ell G d h_k)\alpha \\
    &= \beta h_k (dG+Gd) h_\ell\alpha - \beta h_\ell (dG+Gd) h_k
    \alpha\\
    &= \beta h_k (\alpha\beta-\id) h_\ell\alpha - \beta h_\ell
    (\alpha\beta-\id) h_k \alpha\\
    &= f_k f_\ell - f_\ell f_k
  \end{align*}
  (the last equality holds since $h_kh_\ell = h_\ell h_k$ by
  hypothesis) so that $H_{\{k < \ell\}}$ is a homotopy from $f_\ell
  \circ f_k$ to $f_k \circ f_\ell$ as required.

  \smallskip

  Let us now consider the case $s=3$, $S = \{z_1 < z_2 < z_3\}$. We
  can compute incidence numbers:
  \begin{align*}
    [S:\emptyset] & = -1; \\
    [S: S \setminus z_j] & = -(-1)^j && \text{for } j=1,2,3; \\
    [S: z_j] & = (-1)^j && \text{for } j=1,2,3. \\
  \end{align*}
  We thus need to verify that the sum
  \begin{equation}
    \label{eq:sum_trivial}
    \big( d \circ H_S - H_S \circ d \big) + \underbrace{\sum_{j=1}^3
      (-1)^j \cdot \beta h_{z_j}\alpha \circ H_{S \setminus z_j}}_X -
    \underbrace{\sum_{j=1}^3 (-1)^j \cdot H_{S \setminus z_j} \circ
      \beta h_{z_j}\alpha}_Y
  \end{equation}
  \goodbreak
  is trivial. Now $H_S = \beta \circ \sum_{\sigma \in \Sigma_3}
  \sgn(\sigma) h_{\sigma(1)} G h_{\sigma(2)} G h_{\sigma(3)} \circ
  \alpha$ so that\footnote{here and in the sequel we take the liberty
    to write $\sigma(i)$ in place of~$\sigma(z_i)$ and $h_i$ instead
    of~$h_{z_i}$, to simplify notation}
  \begin{multline}
    \label{eq:calc_dH-Hd}
    d \circ H_S - H_S \circ d \\
    = \beta \circ \sum_{\sigma \in \Sigma_3} \sgn(\sigma) \big(
    h_{\sigma(1)} \gmcb{d} G h_{\sigma(2)} G h_{\sigma(3)} -
    h_{\sigma(1)} G h_{\sigma(2)} G \gmcb{d}
    h_{\sigma(3)} \big) \circ \alpha \\
    = \beta \circ \sum_{\sigma \in \Sigma_3} \sgn(\sigma) h_{\sigma(1)}
    \big( \underbrace{(dG+Gd) h_{\sigma(2)} G}_A - \underbrace{G
      h_{\sigma_2} (dG+Gd)}_B \big) h_{\sigma(3)} \circ \alpha \ .
  \end{multline}
  Since $dG+Gd = \alpha\beta - \id$ we get a contribution of
  \begin{equation}
    \label{eq:1}
    \sgn(\sigma) \beta h_{\sigma(1)} (\alpha\beta-\id) h_{\sigma(2)} G
    h_{\sigma(3)} \alpha
  \end{equation}
  from the term called~$A$ in~\eqref{eq:calc_dH-Hd} above. Now note
  that there are precisely two permutation with $\sigma(3) = j$ a
  fixed value, with opposite signum; since $h_{\sigma(1)}
  h_{\sigma(2)} = h_{\sigma(2)} h_{\sigma(1)}$ this means that the
  factors coming from~``$\id$'' in~\eqref{eq:1} cancel pairwise when
  summing up over all permutations~$\sigma$. The remaining $3!=6$
  terms are
  \begin{align*}
    &   \beta(h_2 \alpha\beta h_3 G h_1 - h_3 \alpha\beta h_2 G h_1) \alpha && \text{(for \(\sigma(3)=1\))}\\
    - & \beta(h_1 \alpha\beta h_3 G h_2 - h_3 \alpha\beta h_1 G h_2) \alpha && \text{(for \(\sigma(3)=2\))}\\
    + & \beta(h_1 \alpha\beta h_2 G h_3 - h_2 \alpha\beta h_1 G h_3) \alpha && \text{(for \(\sigma(3)=3\))}
  \end{align*}
  while the expanded sum of the expression~$X$
  in~\eqref{eq:sum_trivial} reads
  \begin{align*}
    - & \beta h_1 \alpha\beta (h_2 G h_3 - h_3 G h_2) \alpha    && \text{(\(j=1\))}\\
    + & \beta h_2 \alpha\beta (h_1 G h_3 - h_3 G h_1) \alpha    && \text{(\(j=2\))}\\
    - & \beta h_3 \alpha\beta (h_1 G h_2 - h_2 G h_1) \alpha\ . && \text{(\(j=3\))}
  \end{align*}
  That is, after summing up over all permutations~$\sigma$ the terms
  coming from~$A$ and from~$X$ cancel each other.

  Similarly, using $dG+Gd = \alpha\beta - \id$ and summing up over all~$\sigma$
  gives a non-vanishing contribution of
  \begin{align*}
    &   \beta(h_2 G h_3 \alpha\beta h_1 - h_3 G h_2 \alpha\beta h_1)\alpha && \text{(for \(\sigma(3)=1\))}\\
    - & \beta(h_1 G h_3 \alpha\beta h_2 - h_3 G h_1 \alpha\beta h_2)\alpha && \text{(for \(\sigma(3)=2\))}\\
    + & \beta(h_1 G h_2 \alpha\beta h_3 - h_2 G h_1 \alpha\beta h_3)\alpha && \text{(for \(\sigma(3)=3\))}
  \end{align*}
  from~$B$ in~\eqref{eq:calc_dH-Hd}, and expanding~$Y$
  from~\eqref{eq:sum_trivial} yields
  \begin{align*}
    - & \beta(h_2Gh_3 - h_3Gh_2) \alpha\beta h_1 \alpha\\
    + & \beta(h_1Gh_3 - h_3Gh_1) \alpha\beta h_2 \alpha\\
    - & \beta(h_1Gh_2 - h_2Gh_1) \alpha\beta h_3 \alpha
  \end{align*}
  cancelling the previous non-trivial contributions from~$B$. This
  finishes the case $s=3$.

  \smallskip

  We now turn our attention to the case $s \geq 4$, $S = \{z_1 < z_2 <
  \ldots < z_s\}$. We need to check that the sum
  \begin{subequations}
    \label{eq:3}
    \begin{align}
      & (-1)^s [S:\emptyset] \cdot d \circ H_S + [S:\emptyset] \cdot
      H_S \circ d \label{eq:3a} \\
      & \qquad + \sum_{z \in S} [S:S \setminus z][S \setminus z:\emptyset] \cdot
      f_z \circ H_{S \setminus z} \label{eq:3b} \\
      & \qquad\qquad + (-1)^s \sum_{z \in S} [S:z] \cdot H_{S
        \setminus z} \circ f_z \label{eq:3c} \\
      & \qquad \qquad \qquad + \sum_{\substack {\udot T \subseteq S \\ t \geq 2 \leq s-t}}
      (-1)^{ts} [S:T][T:\emptyset] \cdot H_{S \setminus T} \circ H_T \label{eq:3d}
    \end{align}
  \end{subequations}
  (where $t = \# T$ in the last sum) is trivial. Using the definition
  of~$H_S$ in~\eqref{eq:3a} and introducing pairwise cancelling terms
  of the type\begin{multline*} (-1)^{\ell+1} \beta \big( h_{\sigma(1)}
    G h_{\sigma(2)} \ldots G \gmcb{d} h_{\sigma(\ell)} G
    h_{\sigma(\ell+1)} \ldots G h_{\sigma(n)} \\
    - h_{\sigma(1)} G h_{\sigma(2)} \ldots G h_{\sigma(\ell)}
    \gmcb{d} G h_{\sigma(\ell+1)} \ldots G h_{\sigma(n)} \big) \alpha
  \end{multline*}
  (where we write $h_i$ instead of~$h_{z_i}$ as before) we see that up
  to the factor $(-1)^s [S:\emptyset]$, \eqref{eq:3a} is the sum over
  all $\sigma \in \Sigma_s$ of
  \begin{equation}
    \label{eq:4}
    \left.\begin{aligned}
        \sgn(\sigma) \cdot \beta \Big(& h_{\sigma(1)} \gmcb{(dG+Gd)}
        h_{\sigma(2)} G h_{\sigma(3)} G \ldots G h_{\sigma(s)} \\
        -& h_{\sigma(1)} G h_{\sigma(2)} \gmcb{(dG+Gd)} h_{\sigma(3)} G \ldots G
        h_{\sigma(s)} \\
        +& h_{\sigma(1)} G h_{\sigma(2)} G h_{\sigma(3)} \gmcb{(dG+Gd)} \ldots G
        h_{\sigma(s)} \\
        & \qquad\qquad \qquad \qquad \qquad \quad \ddots \\
        + (-1)^s & h_{\sigma(1)} G h_{\sigma(2)} G h_{\sigma(3)} G \ldots
        h_{\sigma(s-1)} \gmcb{(dG+Gd)} h_{\sigma(s)} \Big) \alpha \ .
      \end{aligned}\quad\right\}
  \end{equation}
  Now $dG+Gd = \alpha\beta - \id$; since the maps~$h_j$ commute, all
  contributions from ``$\id$'' will cancel each other upon summing up
  over~$\sigma$. More precisely, by multiplying out the $\ell$th
  summand of~\eqref{eq:4} becomes, up to the sign $(-1)^{\ell+1} \cdot
  \sgn(\sigma)$,
  \begin{gather}
    \beta h_{\sigma(1)} G h_{\sigma(2)}G \ldots \gmcb{h_{\sigma(\ell)}
      \alpha\beta h_{\sigma(\ell+1)}} \ldots G h_{\sigma(s)}\alpha
    \qquad\qquad\qquad
    \tag{\ref{eq:4}$_{\alpha\beta}$}\\
    \qquad \qquad - \beta h_{\sigma(1)} G h_{\sigma(2)}G \ldots
    \gmcb{h_{\sigma(\ell)} h_{\sigma(\ell+1)}} \ldots G h_{\sigma(s)}
    \alpha \ ,\tag{\ref{eq:4}$_\id$} \label{eq:4_id}
  \end{gather}
  and the summand~\eqref{eq:4_id} cancels out with the corresponding
  one coming from the permutation $\sigma \circ (\ell, \ell+1)$.
  
  So we are left with the following non-trivial contribution
  from~\eqref{eq:4}, for a fixed permutation~$\sigma$:
  \begin{align}
    (-1)^s [S:\emptyset] \cdot \sgn(\sigma) \beta \Big(\, &
    h_{\sigma(1)} \gmcb{\alpha\beta} h_{\sigma(2)} G h_{\sigma(3)}
    G \ldots h_{\sigma(s-1)}G h_{\sigma(s)} \tag{\ref{eq:4}$_1$} \label{m1}\\
    - & h_{\sigma(1)} G h_{\sigma(2)} \gmcb{\alpha\beta} h_{\sigma(3)} G \ldots
    h_{\sigma(s-1)} G h_{\sigma(s)} \tag{\ref{eq:4}$_2$} \label{m2}\\
    + & h_{\sigma(1)} G h_{\sigma(2)} G h_{\sigma(3)} \gmcb{\alpha\beta} \ldots
    h_{\sigma(s-1)} G h_{\sigma(s)} \tag{\ref{eq:4}$_3$} \label{m3}\\
    & \hphantom{h_{\sigma(1)} G h_{\sigma(2)} G h_{\sigma(3)}
      \gmcb{\alpha\beta}} \ddots \tag*{$\vdots$\quad}\\
    +(-1)^s& h_{\sigma(1)} G h_{\sigma(2)} G h_{\sigma(3)} G \ldots
    h_{\sigma(s-1)} \gmcb{\alpha\beta} h_{\sigma(s)} \,\Big) \alpha
    \tag{\ref{eq:4}$_{s-1}$} \label{ms-1}
  \end{align}
  We will proceed by expanding the summands~(\ref{eq:3}b--d) and show
  that there is a bijective, sign-reversing correspondence of the
  resulting terms with the
  terms~(\ref{eq:4}$_{1}$--\ref{eq:4}$_{s-1}$) just calculated.

  \smallskip

  Let us begin with~\eqref{eq:3b}. The maps $H_{S \setminus z}$ are
  given as sums of $(s-1)!$ terms indexed by the permutations of~$S
  \setminus z$, while $z$~itself varies over over the set~$S$;
  expanding results in $s \cdot (s-1)! = s!$ terms. More explicitly,
  writing $S = \{z_1 < z_2 < \ldots < z_s\}$ these $s \cdot (s-1)!$
  summands are displayed as
  \begin{multline*}
    \sum_{k=1}^s \sum_{\tau \in \Sigma(S \setminus z_k)} [S:S
    \setminus z_k][S \setminus z_k:\emptyset] \cdot \sgn(\tau) \\
    \cdot \beta h_k \gmcb{\alpha\beta} h_{\tau(1)} G h_{\tau(2)} G \ldots G
    h_{\tau(k-1)} G h_{\tau(k+1)} G \ldots G h_{\tau(s)} \alpha\ ,
  \end{multline*}
  writing $h_i$ for~$h_{z_i}$ and $\tau(i)$ for~$\tau(z_i)$ as before.
  Individually they will correspond, in a bijective manner, to the
  $s!$~different summands~\eqref{m1}, which are indexed by the
  permutation~$\sigma$.

  In detail, given a permutation $\sigma \in \Sigma(S)$ let $k$~be the
  index determined by $z_k = \sigma(z_1)$. Let $\tau \in \Sigma(S
  \setminus z_k)$ be defined by
  \begin{equation}
    \tau(z_i) =
    \begin{cases}
      \sigma(z_{i+1}) & \text{for } i < k \ , \\
      \sigma(z_i) & \text{for } i > k \ .
    \end{cases}\label{eq:seenbefore_1}
  \end{equation}
  Then by construction of~$\tau$ we have an equality of $s$\nbd-tuples
  \begin{multline*}
    \big(\sigma(z_1),\, \sigma(z_2),\, \cdots,\, \sigma(z_s) \big) \\
    = \big( z_k,\, \tau(z_1),\, \tau(z_2),\, \cdots, \tau(z_{k-1}),\,
    \tau(z_{k+1}),\, \cdots,\, \tau(z_s) \big)
  \end{multline*}
  so that the summand of~\eqref{eq:3b} indexed by $z=z_k$ and $\tau
  \in S \setminus z$ agrees with~\eqref{m1} up to sign, and it remains
  to show that the signs are different.

  Now the sign occurring in~\eqref{eq:3b} is $[S:S \setminus z_k] [S
  \setminus z_k : \emptyset] \cdot \sgn(\tau)$. Let us relate
  $\sgn(\tau)$ to~$\sgn(\sigma)$ now. The set of
  inversions\footnote{an inversion of a permutation~$\pi$ of a totally
    ordered finite set is a pair of elements~$(x,y)$ of its domain
    with $x<y$ and $\pi(x)>\pi(y)$} of~$\tau$ is
  \[\mathrm{Inv}(\tau) = \{ (z_i,\, z_j) \,|\, i,j \neq k;\ i<j;\
  \tau(z_i) > \tau(z_j)\} \ .\] Similarly, the set of inversions
  of~$\sigma$ is given by
  \[\mathrm{Inv}(\sigma) = \{ (z_i,\, z_j) \,|\, i<j;\
  \sigma(z_i) > \sigma(z_j)\} \ ,\] and the former injects into the
  latter by a map induced from the assignment $t \mapsto t+1$ if
  $t<k$, and $t \mapsto t$ if $t > k$. The inversions of~$\sigma$ not
  in the image of this injection are precisely the inversions of the
  type~$(z_1,z_j)$, which are parametrised by those $j>1$ satisfying
  $\sigma(z_j) < \sigma(z_1) = z_k$. But as $j$~varies over all
  integers between~$2$ and~$s$, the image $\sigma(z_j)$ varies over
  all of~$S \setminus z_k$ so that the number of such inversions is
  the number of elements of~$S$ strictly less than~$z_k$, of which
  there are~$k-1$. In total, $\# \mathrm{Inv}(\sigma) = \#
  \mathrm{Inv}(\tau) + k-1$. We thus have
  \begin{equation}
    \begin{aligned}
      \sgn(\sigma) &= (-1)^{\mathrm{Inv}(\sigma)} \\
      &= (-1)^{\mathrm{Inv}(\tau)} \cdot (-1)^{k-1} \\
      &= \sgn(\tau) \cdot [S:S \setminus z_k] \ ,
    \end{aligned}\label{eq:seenbefore_2}
  \end{equation}
  the last equality courtesy of Lemma~\ref{lem:interpretation}. Thus
  the total sign occurring in~\eqref{eq:3b} is
  \begin{multline*}
    [S:S \setminus z_k] [S \setminus z_k : \emptyset] \cdot \sgn(\tau) 
    = [S \setminus z_k : \emptyset] \cdot \sgn(\sigma) \\ =
    (-1)^{(s-1)(s-2)/2} \cdot \sgn(\sigma)
  \end{multline*}
  while the sign occurring in~\eqref{m1} is
  \begin{equation*}
    (-1)^s [S:\emptyset] \cdot \sgn(\sigma) = (-1)^s \cdot
    (-1)^{s(s-1)/2} \cdot \sgn(\sigma) = (-1)^{s(s+1)/2} \cdot
    \sgn(\sigma) \ .
  \end{equation*}
  It remains to observe that
  \[\frac{(s-1)(s-2)}2 = \frac{s^2}2 - \frac{3s}2 + 1 \quad \text{and}
  \quad \frac{s^2}2 + \frac s2 = \frac{s(s+1)}2\] differ
  by~$2s-1$ and thus have different parity. This implies that
  \[[S:S \setminus z_k] [S \setminus z_k : \emptyset] \cdot \sgn(\tau)
  = - \big( (-1)^s [S:\emptyset] \cdot \sgn(\sigma) \big)\] as
  required.

  \smallskip

  Not surprisingly, a similar argument works for~\eqref{eq:3c} which is
  linked to the term~\eqref{ms-1}.   Write $S = \{z_1 < z_2 < \ldots <
  z_s\}$ as usual; expanding~\eqref{eq:3c} results in the $s \cdot
  (s-1)!$\nbd-term sum
  \begin{multline*}
    \sum_{k=1}^s \sum_{\tau \in \Sigma(S \setminus z_k)} (-1)^s \cdot
    [S:z_k] \cdot \sgn(\tau) \\
    \cdot \beta h_{\tau(1)} G h_{\tau(2)} G \ldots G h_{\tau(k-1)} G
    h_{\tau(k+1)} G \ldots G h_{\tau(s)} \gmcb{\alpha\beta} h_k \alpha\ ,
  \end{multline*}
  again writing $h_i$ for~$h_{z_i}$ and $\tau(i)$ for~$\tau(z_i)$.
  The summand indexed by~$k$ and $\tau \in \Sigma(Z \setminus z_k)$ is
  uniquely determined by $\sigma \in \Sigma(S)$ as follows: $k$~is the
  index satisfying $z_k = \sigma(z_s)$, and $\tau$~is given by
  \[\tau(z_i) =
  \begin{cases}
    \sigma(z_i)    & \text{for } i < k \ , \\
    \sigma(z_{i-1}) & \text{for } i > k \ .
  \end{cases}\]
  Then by construction of~$\tau$ we have an equality of $s$\nbd-tuples
  \begin{multline*}
    \big(\sigma(z_1),\, \sigma(z_2),\, \cdots,\, \sigma(z_s) \big) \\
    = \big( \tau(z_1),\, \tau(z_2),\, \cdots, \tau(z_{k-1}),\,
    \tau(z_{k+1}),\, \cdots,\, \tau(z_s),\, z_k \big)
  \end{multline*}
  so that the summand of~\eqref{eq:3c} indexed by $z=z_k$ and $\tau
  \in S \setminus z$ agrees with~\eqref{ms-1} up to sign, and it remains
  to show that the signs are different.

  There is an injective map of sets of inversions
  \[\mathrm{Inv}(\tau) \rTo \mathrm{Inv}(\sigma)\] induced by the
  assignment $t \mapsto t$ if $t<k$, and $t \mapsto t-1$ if $t>k$. The
  inversions of~$\sigma$ not in the image are of the type $(i,s)$
  where $\sigma(z_i)> \sigma(z_s) = z_k$. But as $i$ varies over all
  numbers from~$1$ to~$s-1$, the element $\sigma(z_i)$ varies over all
  of~$S \setminus z_k$. That is, the number of inversions of~$\sigma$
  which are not in the image is the number of elements of~$S$ which
  are strictly greater than~$z_k$, of which there are $s-k$ many. In
  other words,
  \[\sgn(\sigma) = \sgn(\tau) \cdot (-1)^{s-k} \ .\] So the sign
  occurring in~\eqref{ms-1} is
  \begin{align*}
    (-1)^s [S:\emptyset] (-1)^s \cdot \sgn(\sigma) &= (-1)^{s(s-1)/2}
    \cdot \sgn(\sigma) \\
    &= (-1)^{s(s-1)/2} \cdot (-1)^{s-k} \cdot \sgn(\tau) \\
    &= (-1)^{s(s+1)/2-k} \cdot \sgn(\tau) \ .
  \end{align*}
  
  On the other hand, by Lemma~\ref{lem:interpretation} we know that
  $[S:z_k] = (-1)^\kappa$, where~$\kappa$ is the number of pairs
  $(z_i,z_j) \in S \times S$ with $i<j$ and $j \neq k$; that is,
  $\kappa$~is the number of unordered pairs of elements of~$S$,
  reduced by the number of elements in~$S$ strictly less than~$z_k$:
  \[\kappa = \frac12s(s-1) - (k-1)\] This means that the sign
  occurring in~\eqref{eq:3c} is
  \begin{align*}
    (-1)^s \cdot [S:z_k] \cdot \sgn(\tau) &= (-1)^s \cdot
    (-1)^{s(s-1)/2-k+1} \cdot \sgn(\tau) \\
    &= - (-1)^{s(s+1)/2-k} \cdot \sgn(\tau) \ ,
  \end{align*}
  which is opposite to the above as required.

  \smallskip

  It remains to deal with~(\ref{eq:4}$_\ell$), for $1 < \ell < s-1$,
  and relate these to the summands of~\eqref{eq:3d}. For reference, we
  record that the latter expands to
  \begin{equation}
    \left.
      \begin{gathered}
        \sum_{\substack {\udot T \subseteq S \\ t \geq 2 \leq s-t}}
        \sum_{\tau \in \Sigma(S \setminus T)} \sum_{\mu \in \Sigma(T)}
        (-1)^{ts} [S:T][T:\emptyset] \cdot \sgn(\tau) \cdot \sgn(\mu)
        \qquad \\
        \qquad \cdot \beta h_{\tau(1)} G h_{\tau(2)} G \ldots G h_{\tau(\ell)}
        \gmcb{\alpha\beta} h_{\mu(1)} G h_{\mu(2)} G \ldots G
        h_{\mu(s-\ell)} \alpha \ ,
      \end{gathered}
    \quad\right\}
    %\tag*{(\ref{eq:3d})\(_{\text{exp}}\)}
    \label{eq:3d_expanded}
  \end{equation}
  where we have written $\ell = s-t$, $S\setminus T = \{x_1 < x_2 <
  \ldots < x_\ell\}$ and $T = \{y_1 < y_2 < \ldots < y_{s-\ell}\}$,
  and also used the abbreviations $\tau(i) = \tau(x_i)$ and $\mu(j) =
  \mu(y_j)$.
  
  \smallskip

  For $S = \{z_1 < \ldots < z_s\}$ as usual, define $T = \{
  \sigma(\ell+1),\, \sigma(\ell+2),\, \cdots,\, \sigma(s) \}$ so that
  $t = \# T = s-\ell$. Rename the elements of $S \setminus T = \{
  \sigma(1),\, \cdots,\, \sigma(\ell)\}$ as $\{x_1 < x_2 < \ldots <
  x_\ell\}$, and let $\tau \in \Sigma(S \setminus T)$ be determined by
  $\tau(x_j) = \sigma(z_j)$ for $1 \leq j \leq \ell$. That is,
  $\tau$~is the unique permutation of~$S \setminus T$ satisfying
  \begin{equation}
    \label{eq:tau}
    \tau\inv \sigma(z_1) < \tau\inv \sigma(z_2) < \ldots < \tau\inv
    \sigma(z_\ell) \ .
  \end{equation}
  Similarly, let $\mu \in \Sigma(T)$ be the unique permutation with
  \begin{equation}
    \label{eq:mu}
    \mu\inv \sigma(z_{\ell +1}) < \mu\inv \sigma(z_{\ell+2}) < \ldots
    < \mu\inv \sigma(z_s) \ .
  \end{equation}
  By construction the summand of~\eqref{eq:3d} indexed by $T$, $\tau$
  and~$\mu$ agrees with~(\ref{eq:4}$_\ell$) up to sign, and this
  correspondence is uniquely reversible.

  It remains to check that the signs are different. To this end, let
  $\nu = (\tau\inv,\, \mu\inv) \circ \sigma \in \Sigma(S)$ where we
  use the canonical embedding
  \[(\tau\inv,\, \mu\inv) \in \Sigma(S \setminus T) \times \Sigma(T)
  \subseteq \Sigma(S)\] coming from the decomposition $S = (S
  \setminus T) \amalg T$. Then on the one hand,
  \begin{equation}
    \label{eq:nu}
    \sgn(\nu) = \sgn (\tau) \cdot \sgn(\mu) \cdot \sgn(\sigma) \ .
  \end{equation}
  On the other hand, in view of~\eqref{eq:tau} and~\eqref{eq:mu} we
  know that the inversions of~$\nu$ are parametrised by the pairs
  $(i,j)$ with $i \leq \ell < j$ and $\nu(z_i) > \nu(z_j)$. Now the
  elements of the form $\nu(i)$, for $1 \leq i \leq \ell$, are
  precisely the elements of the form $\sigma(i)$, for $i$~in the same
  range, and are thus precisely the elements of~$S \setminus T$;
  similarly, the elements of the form $\nu(j)$, for $\ell < i \leq s$
  are precisely the elements of~$T$. In other words, the
  number~$\lambda$ of inversions of~$\nu$ is the number of pairs
  $(y,x) \in T \times (S \setminus T)$ with $y < x$.

  Let $\omega$~be the number of pairs $(y,x) \in (S \setminus T)
  \times (S \setminus T)$ with $y < x$. Then
  \[\omega = \frac12 \ell (\ell-1) \ ,\] by direct counting. Clearly
  $\kappa = \lambda + \omega$ is the number of pairs $(y,x) \in S
  \times (S \setminus T)$ with $y<x$. But $(-1)^\kappa = [S:T]$, by
  Lemma~\ref{lem:interpretation}; combined with~\eqref{eq:nu} this
  yields
  \begin{multline*}
    [S:T] = (-1)^\omega \cdot (-1)^\lambda = (-1)^{\ell(\ell-1)/2}
    \cdot \sgn(\nu) \\ = (-1)^{\ell(\ell-1)/2} \cdot \sgn (\tau) \cdot
    \sgn(\mu) \cdot \sgn(\sigma) \ .
  \end{multline*}

  So the sign occurring in~\eqref{eq:3d}, or rather in its expanded
  form~\eqref{eq:3d_expanded}, is given by
  \begin{multline*}
    (-1)^{ts} [S:T][T:\emptyset] \cdot \sgn(\tau) \cdot \sgn(\mu) \\ =
    (-1)^{(s-\ell)s} \cdot (-1)^{\ell (\ell-1)/2}[T:\emptyset] \cdot
    \sgn(\sigma) \qquad \qquad \qquad \qquad\\
    = (-1)^{(s-\ell)s} \cdot (-1)^{\ell (\ell-1)/2}[T:\emptyset] \cdot
    (-1)^{(s-\ell)(s-\ell-1)/2} \cdot \sgn(\sigma) \ ,
  \end{multline*}
  and a calculation we omit shows that this agrees with
  \[(-1)^{\ell + s(s+1)/2} \cdot \sgn(\sigma)\ .\] The sign occurring
  in~(\ref{eq:4}$_\ell$) is the opposite as
  \begin{multline*}
    (-1)^s [S:\emptyset] \cdot \sgn(\sigma) \cdot (-1)^{\ell+1} =
    (-1)^s \cdot (-1)^{s(s-1)/2} \cdot (-1)^{\ell+1} \cdot
    \sgn(\sigma) \\ = - (-1)^{\ell + s(s+1)/2} \cdot \sgn(\sigma) \ ,
  \end{multline*}
  just as required. This finishes the proof of
  Theorem~\ref{thm:main_cube}.
\end{proof}

\part{Mapping tori and the Mather trick}
\label{part:tori}

\section{Filtration}
\label{sec:filtration}

We now record that the totalisation construction introduced in
Definition~\ref{def:totalisation} comes with a natural filtration by
cardinality of the index set. That is, given a homotopy commutative
$N$\nbd-cube~$F$, cf.~Definition~\ref{def:homotopy_commutative},
define for fixed $k \in \bZ$ the modules
\[\tot_k(F)^n:= \bigoplus_{\substack{A \subseteq N \\ \#A \geq k}}
F(A)^{n-a}\] together with module homomorphisms
\[D(F)_k = D_k = (D_{k,B,A})_{A,B \subseteq N} \colon \tot_k(F)^n \rTo
\tot_k(F)^{n+1}\] given by \[D_{k,B,A} = \begin{cases} (-1)^{ab} [B:A]
  \cdot H_{B,A} & \text{if } A \subseteq B \text{ and } \#A \geq k \ ,
  \\ 0 & \text{otherwise.} \end{cases}\]

\begin{proposition}
  \label{prop:filtration}
  The graded module $\tot_k(F)$ together with~$D(F)_k$ is a cochain
  complex. We have a descending filtration of subcomplexes
  \[\tot(F) = \tot_0 (F) \supseteq \tot_1(F) \supseteq \ldots
  \supseteq \tot_n(F) \supseteq \tot_{n+1}(F) = 0 \ .\] The filtration
  quotients are
  \[\tot_k(F) / \tot_{k+1}(F) = \bigoplus_{\substack{A \subseteq N \\
      \#A = k}} \Sigma^k F(A) \ ,\] that is, are the direct sums of
  the cochain complexes $F(A)$ with $\# A = k$,
  suspended\/\footnote{The $k$th suspension of cochain complexes
    modifies the differential by a factor of~$(-1)^k$.} $k$~times.\qed
\end{proposition}

\section{The Mather trick for derived special $N$-cubes}
\label{sec:mapping_torus_analogue}

Let $h_1,\, \cdots,\, h_n$ be pairwise commuting self-maps of the
chain complex~$D$. As mentioned at the beginning
of~\S\ref{sec:explicit}, we can form the trivial special $N$\nbd-cube
$\triv(D;\, h_1,\, h_2,\, \cdots,\, h_n)$ (with trivial homo\-topies
and higher homotopies); for convenience we denote its totalisation
by~$X$ and its differential by $D^X = (D^X_{B,A})_{A,B \subseteq N}$.

Let $g \colon D \rTo D$ be a cochain map, and let $G \colon \id \simeq
g$ be a homotopy so that $dG+Gd= g-\id$. We can now form the derived
special $N$\nbd-cube $\derived(D;\, g,\, \id_D,\, G;\, h_1,\, h_2,\,
\cdots,\, h_n)$ according to Theorem~\ref{thm:main_cube} and
Definition~\ref{def:derived_cube}, with $\beta = \id_D$ and $\alpha =
g$. We retain the notation of~\ref{thm:main_cube}: we have chain maps
$f_k = h_kg$ and higher homotopies~$H_S$ (for $s \geq 2$), and maps
$H_{B,A} = H_{B \setminus A}$.  We denote the totalisation by~$Y$, and
write $D^Y = (D^Y_{B,A})_{A,B \subseteq N}$ for its differential.

Finally, we define a matrix $M = (M_{B,A})_{A,B \subseteq N}$ which,
{\it a priori\/}, is just a collection of module homomorphisms $Y^n
\rTo X^n$. We set $M_{B,A} = 0$ if $A \not\subseteq B$. Otherwise, we
set
\begin{equation}
  \label{eq:def_M_BA}
  M_{B,A} = (-1)^b(-1)^{ab} [B:A] \cdot M_{B \setminus A} \ ,
\end{equation}
where $a = \# A$ and $b = \# B$ as usual, and
\begin{equation}
  \label{eq:def_M_S}
  \left.
    \begin{aligned}
      M_\emptyset &= g \ , \\
      M_{\{k\}} &= G \circ f_k = G \circ h_k \circ g && \text{for \(k \in
        S\)}, \\
      M_S &= G \circ H_S && \text{for \(s \geq 2\)}.
    \end{aligned}
    \quad\right\}
\end{equation}
Writing $S = \{z_1 < z_2 < \ldots < z_s\}$ and
abbreviating $h_{z_i}$ by~$h_i$ as before 
we have more explicitly, for $s \geq 2$,
\[M_S = G \circ \sum_{\sigma \in \Sigma_s} \sgn(\sigma) \cdot
h_{\sigma(1)} G h_{\sigma(2)} G \ldots G h_{\sigma(s)} \circ g \ .\]

\begin{lemma}
  \label{lem:compare_triv_analogue}
  The matrix $M$ defines a cochain map
  \begin{multline*} 
    Y = \tot \derived(D;\, g,\, \id_D,\, G;\, h_1,\,  h_2,\, \cdots,\,
    h_n) \\
    \rTo^M \tot \triv(D;\, h_1,\, h_2,\, \cdots,\, h_n) = X \ .
  \end{multline*}
  This map is a quasi-isomorphism. If $D$ is a bounded above complex
  of projective $R$\nbd-modules, then $M$ is a homotopy equivalence $M
  \colon Y \simeq X$.
\end{lemma}

\begin{proof}
  Granting that $M$~defines a cochain map, it is almost trivial to
  verify that $M$ respects the filtration of the totalisation, such
  that the induced map on filtration quotients
  \[Y_k / Y_{k+1} \rTo X_k / X_{k+1}\] is a direct sum of the ($k$th
  suspension of) the cochain map~$g$. This implies, by an iterative
  application of the five lemma for $k= n, n-1, \cdots, 0$, that $M
  \colon Y_k \rTo X_k$ is a quasi-isomorphism. In particular, $M$~is a
  quasi-isomorphism $Y = Y_0 \rTo X_0 = X$. The final claim about
  bounded above complexes follows from general homological algebra.

  We now check that $M$~defines a cochain map; to this end we need to
  verify that $MD^Y = D^X M$, that is, using the fact that all
  matrices under consideration are triangular,
  \[\big( MD^Y \big)_{B,A} - \big(D^X M \big)_{B,A} = 0\] for all $A
  \subseteq B \subseteq N$.

  We will analyse the second summand $\big( D^X M \big)_{B,A} =
  \sum_{B \supseteq \udot S \supseteq A} D^X_{B,S} M_{S,A}$ first, and
  assume that $b-a \geq 3$. As $X$~is the totalisation of a trivial
  $N$\nbd-cube we know that $D^X_{B,S} = 0$ unless $\#(B \setminus S)
  \leq 1$; this means that
  \begin{multline*}
    \big( D^X M \big)_{B,A} = (-1)^b d \,\cdot\, (-1)^b (-1)^{ab}
    [B:A] \cdot G \circ H_{B \setminus A} \\
    + \sum_{z \in B \setminus A} [B : B \setminus z] \cdot h_z \circ
    (-1)^{b-1} (-1)^{a(b-1)} [B \setminus z : A] \cdot G \circ H_{(B
      \setminus z) \setminus A}
  \end{multline*}
  so that, simplifying the sign terms and introducing a term of the
  form $GdH_{B \setminus A} - GdH_{B \setminus A}$, we obtain
  \begin{multline*}
    (-1)^{ab} \big( D^X M \big)_{B,A} = [B:A] \big( (dG+Gd) \circ
    H_{B \setminus A} - GdH_{B \setminus A} \big) \\
    - (-1)^{b-a} \cdot \sum_{z \in B \setminus A} [B : B \setminus z]
    [B \setminus z : A] \cdot h_z \circ G \circ H_{(B \setminus z)
      \setminus A} \ .
  \end{multline*}
  The second line of this expression reduces to $[B:A] \cdot H_{B
    \setminus A}$, as we will verify presently; as $dG+Gd = g - \id$
  this means that
  \begin{equation}
    \label{eq:DXMBA}
    \big( D^X M \big)_{B,A} = (-1)^{ab} [B:A] \cdot ( g H_{B \setminus
      A} - GdH_{B \setminus A}) \ .
  \end{equation}
  To verify the claim, let us use the explicit definition of~$H_{(B
    \setminus z) \setminus A}$; writing $B \setminus A = {z_1 < z_2 <
    \ldots < z_{b-a}}$ and abbreviating~$h_{z_k}$ by~$h_k$ as usual,
  we have
  \begin{multline*}
    - \sum_{z \in B \setminus A} [B : B \setminus z] [B \setminus z :
    A] \cdot h_z \circ G \circ H_{(B \setminus z) \setminus A}
    \\ 
    = \sum_{k=1}^{b-a} \sum_{\tau \in \Sigma(B \setminus z_k)} - [B :
    B \setminus z_k] [B \setminus z_k : A] \cdot \sgn(\tau) \qquad
    \qquad \\
    \cdot h_k G h_{\tau(1)} G h_{\tau(2)} G \ldots G h_{\tau(k-1)} G
    h_{\tau(k+1)} G \ldots G h_{\tau(b)} g\ ,
  \end{multline*}
  using $\tau(i)$ for~$\tau(z_i)$ as
  before; we need to compare that with
  \[[B:A] \cdot H_{B \setminus A} = \sum_{\sigma \in \Sigma(B
    \setminus A)} [B:A] \cdot \sgn(\sigma) \cdot h_{\sigma(1)} G
  h_{\sigma(2)} G \ldots G h_{\sigma(b)} g \ .\] As we have seen
  before, there is a bijection
  \[\Sigma(B \setminus A) \ni \sigma \mapsto (k, \tau) \text { where }
  z_k = \sigma(z_1) \in B \text{ and } \tau \in \Sigma\big( (B
  \setminus z_k) \setminus A \big) \] such that $\sgn(\sigma) =
  (-1)^{k-1} \sgn(\tau)$, see \eqref{eq:seenbefore_1}
  and~\eqref{eq:seenbefore_2}. So we have reduced to showing
  \[[B:B \setminus z_k][B \setminus z_k:A] = (-1)^{b-a} (-1)^k [B:A] \
  .\] Now by Lemma~\ref{lem:interpretation} the left-hand side is
  given by $(-1)^\alpha \cdot (-1)^\beta$ where
  \begin{align*}
    \alpha &= \# \{ x \in B \,|\, x < z_k\}\ , \\
    \beta  &= \# \{ x<y \,|\, x,y \in B \setminus z_k, x<y, y \notin
    A\} \ ;
  \end{align*}
  the (disjoint) union of the two sets is the same as
  \begin{multline*}
    \{x<y \,|\, x,y \notin A\} \amalg \{x<y \,|\, x \in A, y \notin
    A\} \setminus \{y \,|\, z_k<y, y \notin A\} \\
    = \{x<y \,|\, x \in B, y \in B \setminus A\} \setminus \{y \,|\,
    z_k<y, y \notin A\} \ ;
  \end{multline*}
  the set taken away has clearly $b-a-k$ elements so that, using
  Lemma~\ref{lem:interpretation} again, we have
  \begin{multline*}
    [B:B \setminus z_k][B \setminus z_k:A] = (-1)^\alpha \cdot
    (-1)^\beta \\ = [B:A] (-1)^{b-a-k} = (-1)^{b-a} (-1)^k [B:A]
  \end{multline*}
  just as required.

  \smallbreak

  We now need to consider the explicit form of \[(MD^Y)_{B,A} =
  \sum_{B \supseteq \udot S \supseteq A} M_{B,S} D^Y_{S,A} \ ;\]
  plugging in all relevant definitions, and using~\eqref{eq:DXMBA}, we
  see that the difference $(MD^Y)_{B,A} - (D^XM)_{B,A}$ is given by
  the expression
  \begin{multline}
    \label{eq:MDY-DXM}
    g \circ (-1)^{ab} [B:A] \cdot H_{B \setminus A}\\
    \noalign{\smallskip}
     + \sum_{z \in B \setminus A} (-1)^b [B:B \setminus z] \cdot Gf_z
    \circ (-1)^{(b-1)a} [B \setminus z:A] \cdot H_{(B \setminus
      z) \setminus A} \\
     + \hskip -1 em \sum_{\substack{B \supseteq \udot T \supseteq A \\
        b-t \geq 2 \leq t-a}} \hskip -1 em (-1)^b (-1)^{bt} [B:T]
    \cdot GH_{B \setminus T} \circ (-1)^{at} [T:A] \cdot
    H_{T \setminus A} \\
     + \sum_{z \in B \setminus A} (-1)^b (-1)^{(a+1)b} [B:A \amalg
    z] \cdot GH_{B \setminus (A \amalg z)} \circ [A \amalg z:A] \cdot f_z \\
     + (-1)^b (-1)^{ab} [B:A] GH_{B \setminus A} \circ (-1)^a d \\
    \noalign{\smallskip}
     - (-1)^{ab} [B:A] \cdot ( g H_{B \setminus A} -GdH_{B \setminus
      A}) \ .
  \end{multline}
  To analyse the sign in the third line of this expression, the sum
  over~$T$, note that the total exponent of~$-1$ is
  \[b + bt + at \equiv b + (t-a)(b-a)+a(b-a) \equiv (b-a) + ab +
  (t-a)(b-a)\] modulo~$2$. That is, $(-1)^{ab}(-1)^{b-a} \cdot
  \big((MD^Y)_{B,A} - (D^XM)_{B,A}\big)$ equals
  \begin{multline*}
    (-1)^{b-a} [B:A] \cdot g \circ H_{B \setminus A}\\
    + \sum_{z \in B \setminus A} [B:B \setminus z] [B \setminus
    z:A] \cdot G \circ f_z \circ H_{(B \setminus z) \setminus A} \\
    + \sum_{\substack{B \supseteq \udot T \supseteq A \\ 
        b-t \geq 2 \leq t-a}} \hskip -1 em (-1)^{(t-a)(b-a)} [B:T] [T:A]
    \cdot G \circ H_{B \setminus T} \circ H_{T \setminus A} \\ 
    + (-1)^{b-a} \cdot \sum_{z \in B \setminus A} [B:A \amalg z] [A
    \amalg z:A] \cdot G \circ H_{B \setminus (A \amalg z)} \circ f_z \\
    + [B:A] \cdot G \circ H_{B \setminus A} \circ d \\
    + (-1)^{b-a} [B:A] \cdot (GdH_{B \setminus A} - g H_{B \setminus A}) \ .
  \end{multline*}
  We note that the very first summand and the very last (after
  multiplying out the last pair of parentheses) cancel. We re-write
  the remaining terms, setting $S = B \setminus A$ so that $s =
  b-a$. We also let $T$ be a subset of~$S$ so that the role of~$T$
  above is now played by~$T \amalg A$, and $t$ has to be replaced
  by~$t+a$. In addition we take out a common post-composition
  with~$G$. All told, we now want to show that the following
  expression is trivial:
  \begin{multline*}
    \sum_{z \in S} [B:B \setminus z] [B \setminus z:A] \cdot f_z
    \circ H_{S \setminus z} \\
    + \sum_{\substack{\udot T \subseteq S \\ s-t \geq 2 \leq s}}
    \hskip -1 em (-1)^{st} [B:A \amalg T] [A \amalg T:A] \cdot
    H_{S \setminus T} \circ H_T \qquad \qquad \quad\\
    + (-1)^s \cdot \sum_{z \in S} [B:A \amalg z] [A \amalg z:A]
    \cdot H_{S \setminus z} \circ f_z \qquad\\
    + [B:A] \cdot H_S \circ d \qquad\qquad\qquad\qquad\qquad\\
    + (-1)^s [B:A] \cdot d \circ H_S
    \qquad\qquad\qquad\qquad\qquad \qquad
  \end{multline*}
  Now compare this expression with the analogous one obtained by
  removing an element $x \in A$ from both~$A$ and~$B$, that is,
  obtained by the substitutions $A \mapsto A \setminus x$ and $B
  \mapsto B \setminus x$. The numbers $t$ and $s = b-a = (b-1)-(a-1)$,
  and the difference set $S = B \setminus A = (B \setminus x)
  \setminus (A \setminus x)$ are clearly unaffected by this
  change. All the products of incidence numbers, as well as the
  incidence number $[B:A] = [B:B][B:A]$ acquire a sign that depends
  on~$A$, $B$ and~$x$ only, by Lemma~\ref{lem:incidence_independent},
  so is the same for all summands. That is, up to sign the whole
  expression remains unchanged, so we may just as well consider the
  case $A= \emptyset$ only. But then the expression is trivial by
  Corollary~\ref{cor:criterion}.

  The verification for $b-a < 3$ can be done along similar lines, but
  is slightly easier due to fewer terms being involved. We omit the
  details.
\end{proof}

We introduce yet more notation: suppose we are given, in addition to
the maps~$h_i$ above, cochain maps $\alpha \colon C \rTo D$ and $\beta
\colon D \rTo C$.  We set $g = \alpha \circ \beta \colon D \rTo D$,
and let $G$ as before be a homotopy $G \colon \id_D \simeq \alpha
\circ \beta = g$ so that $dG+Gd = \alpha\beta-\id$; this is, according
to Theorem~\ref{thm:main_cube}, precisely the data required to define
the special $N$\nbd-cube $\derived(C;\, \alpha,\, \beta,\, G;\,
h_1,\,h_2,\, \cdots,\, h_n)$.  We denote its totalisation by~$Z$, and
the differential by~$D^Z=(D^Z_{B,A})_{A,B \subseteq N}$.

We let $L$ denote the (constant) diagonal matrix with entry $\beta$;
this is to be considered as a module homomorphism $Y^n \rTo Z^n$.

\begin{lemma}[\textsc{Mather} trick for derived special $N$-cubes]
  \label{lem:mather_for_analogue}
  The matrix~$L$ defines a cochain map
  \begin{multline*}
    \tot \derived(D;\, \alpha\beta,\, \id_D,\, G;\, h_1,\, h_2,\,
    \cdots,\, h_n) = Y \\
    \rTo^L Z = \tot \derived(C;\, \alpha,\, \beta,\, G;\, h_1,\,
    h_2,\, \cdots,\, h_n) \ .
  \end{multline*}
  If $\beta$~is a quasi-isomorphism, then $L$ is a quasi-isomorphism
  as well. If in addition $C$ and~$D$ are bounded above complexes of
  projective $R$\nbd-modules, then $L$ is a homotopy equivalence $L
  \colon Y \simeq Z$.
\end{lemma}

\begin{proof}
  It is a straightforward calculation that $LD^Y = D^ZL$ so that
  $L$~defines a cochain map. The map~$L$ respects the filtration of
  totalisation, such that the induced map on filtration quotients
  \[Y_k / Y_{k+1} \rTo Z_k / Z_{k+1}\] is a direct sum of the ($k$th
  suspension of) the cochain map~$\beta$. This implies, by an
  iterative application of the five lemma for $k= n, n-1, \cdots, 0$,
  that $L \colon Y_k \rTo Z_k$ is a quasi-isomorphism. In particular,
  $L$~is a quasi-isomorphism $Y = Y_0 \rTo Z_0 = Z$. The final claim
  about bounded above complexes follows from general homological
  algebra.
\end{proof}

A more explicit treatment of these results
(Lemmas~\ref{lem:compare_triv_analogue}
and~\ref{lem:mather_for_analogue}) in the special case $n=2$ of two
\textsc{Laurent} variables can be found in
\cite[Lemma~III.7.4]{square}, with slightly different sign
conventions.

\section{Higher-dimensional mapping tori}
\label{sec:high-dim-tori}

\begin{definition}
  \label{def:high_dim_tori}
  Let $F$ be a special $N$\nbd-cube on the $R$\nbd-module cochain
  complex~$C$; we think of this as a collection of cochain maps $f_k$
  which commute up to specified coherent homotopy. Write $L =
  R[x_{1}^{\pm 1},\, x_{2}^{\pm 1},\, \cdots,\, x_{n}^{\pm
    1}]$.  We define the {\it mapping $n$\nbd-torus of~$F$}, denoted
  $\torus\, F$, to be the totalisation of the special $N$\nbd-cube
  $\bar F$ on the $L$\nbd-module complex $\bar C = C \tensor_R L$
  specified by the following data:
  \begin{itemize}
  \item differential $\bar d = d_{\bar C} = d_C \tensor 1$;
  \item $\bar f_k = f_k \tensor 1 - 1 \tensor x_k$, for $k \in N$;
  \item $\bar H_S = H_S \tensor 1$, for any $S \subseteq N$ with $s
    \geq 2$.
  \end{itemize}
\end{definition}

\begin{example}
  \label{ex:torus_of_trivial}
  For pairwise commuting self-maps $h_k$ of~$C$, we have
  \begin{multline*}
    \torus\,\triv(C;\, h_1,\, h_2,\, \cdots,\, h_n) = \tot \triv(C
    \tensor_R L;\, h_1 \tensor 1 - 1 \tensor x_1,\,\\ h_2 \tensor 1 -
    1 \tensor x_2,\, \cdots,\, h_n \tensor 1 - 1 \tensor x_n) \ .
  \end{multline*}
\end{example}

\begin{example}
  \label{ex:mapping_2-torus}
  Let $f_{0},f_{1} \colon C \rTo C$ be two cochain maps, and let $H =
  H_{\{0,1\}} \colon f_{1} f_{0}\simeq f_{0} f_{1}$ be a given
  homotopy. These data define a special $\{0<1\}$-cube~$F$,
  and up to shift, sign and naming conventions the totalisation of the
  associated special $\{0<1\}$\nbd-cube~$\bar F$ is the mapping
  $2$\nbd-torus $\mathcal{T} (f_{1},\, f_{0};\, H)$ as discussed by
  the authors in \cite[\S{}III.6]{square}.
\end{example}

Of course we need to check that the definition of mapping tori makes
sense:

\begin{lemma}
  \label{lem:torus_meaningful}
  The data listed in Definition~\ref{def:high_dim_tori} define a
  special $N$\nbd-cube.
\end{lemma}

\begin{proof}
  We verify that the hypotheses of Corollary~\ref{cor:criterion} are
  satisfied. For $k < \ell$ we have
  \begin{align*}
    \bar d \bar H_{\{k<\ell\}} + \bar H_{\{k<\ell\}} \bar d & =
    (dH_{\{k<\ell\}} + H_{\{k<\ell\}}d) \tensor 1 \\
    &= (f_kf_\ell - f_\ell f_k) \tensor 1 \\
    &= (f_k \tensor 1 - 1 \tensor x_k) (f_\ell \tensor 1 - 1 \tensor
    x_\ell) \\
    & \qquad\qquad - (f_\ell \tensor 1 - 1 \tensor x_\ell) (f_k
    \tensor 1 - 1 \tensor x_k) \\
    &= \bar f_k \bar f_\ell - \bar f_\ell \bar f_k
  \end{align*}
  since maps of the form $? \tensor 1$ and $1 \tensor ?$,
  respectively, commute; this shows that $\bar H_{\{k<\ell\}}$ is a
  homotopy from $\bar f_\ell \bar f_k$ to $\bar f_k \bar f_\ell$ as
  required (we have of course used that $H_{\{k<\ell\}}$ is a homotopy
  from $f_\ell f_k$ to~$f_k f_\ell$ as $F$~is a special $N$\nbd-cube.)

  For $S \subset N$ with $ s \geq 3$ we observe that
  \begin{multline*}
    (-1)^s [S:\emptyset] \cdot \bar d \circ \bar H_S + [S:\emptyset] \cdot
    \bar H_S \circ \bar d \\
    + \sum_{z \in S} [S:S \setminus z][S \setminus z:\emptyset] \cdot
    \bar f_z \circ \bar H_{S \setminus z} \\
    + (-1)^s \sum_{z \in S} [S:z] \cdot \bar H_{S \setminus z} \circ \bar f_z \\
    + \sum_{\substack {\udot T \subseteq S \\ t \geq 2 \leq s-t}}
    (-1)^{ts} [S:T][T:\emptyset] \cdot \bar H_{S \setminus T} \circ
    \bar H_T \hphantom{\ ,}
  \end{multline*}
  is the same as
  \begin{multline*}
    \Big((-1)^s [S:\emptyset] \cdot d \circ H_S + [S:\emptyset] \cdot
    H_S \circ d \\
    + \sum_{z \in S} [S:S \setminus z][S \setminus z:\emptyset] \cdot
    f_z \circ H_{S \setminus z} \qquad \qquad \qquad \qquad \qquad
    \qquad \\
    + (-1)^s \sum_{z \in S} [S:z] \cdot H_{S \setminus z} \circ f_z
    \qquad \qquad \qquad \qquad \qquad \\
    \qquad \qquad + \sum_{\substack {\udot T \subseteq S \\ t \geq 2 \leq s-t}}
    (-1)^{ts} [S:T][T:\emptyset] \cdot H_{S \setminus T} \circ H_T
    \Big) \tensor 1 \\
    \noalign{\smallskip}
    - \Big( \sum_{z \in S} \big( [S:S \setminus z][S \setminus
    z:\emptyset] + (-1)^s [S:z] \big) \cdot H_{S
      \setminus z} \tensor x_z
    \Big) \ ;
  \end{multline*}
  to which the first four lines contribute nothing as $F$~is a special
  $N$\nbd-cube, using Corollary~\ref{cor:criterion}. So it is enough
  to verify that
  \[[S:S \setminus z][S \setminus z:\emptyset] + (-1)^s [S:z] = 0\]
  for every $z \in S$, which is a pleasant combinatorial exercise left
  to the interested reader.
\end{proof}

\section{The Mather trick for mapping tori}
\label{sec:mather_trick_for_tori}

Let $h_1,\, \cdots,\, h_n \colon D \rTo D$ be mutually commuting
self-maps of the $R$\nbd-module cochain complex~$D$. We denote the
mapping torus of the trivial special $N$\nbd-cube $\triv(D;\, h_1,\,
h_2,\, \cdots,\, h_n)$ by~$X$, and denote the differential of~$X$ by
$D^X = (D^X_{B,A})_{A,B \subseteq N}$.

Let $g \colon D \rTo D$ be another cochain map, and let $G \colon \id
\simeq g$ be a homotopy so that $dG+Gd= g-\id$. We can now form the
derived special $N$\nbd-cube $\derived(D;\, g,\, \id_D,\, G;\, h_1,\,
h_2,\, \cdots,\, h_n)$ according to Theorem~\ref{thm:main_cube} and
Definition~\ref{def:derived_cube}, with $\beta = \id_D$ and $\alpha =
g$. We denote its mapping torus by~$Y$, which has differential $D^Y =
(D^Y_{B,A})_{A,B \subseteq N}$. The complex $Y$ is defined in terms of
certain maps $\bar d$, $\bar f_k$ and $\bar H_S$ as prescribed in
Definition~\ref{def:high_dim_tori}.

Finally, we define a matrix $K = (K_{B,A})_{A,B \subseteq N}$ which,
{\it a priori\/}, is just a collection of module homomorphisms $Y^n
\rTo X^n$:
\[K_{B,A} = M_{B,A} \tensor 1\]
where $M_{B,A}$ is as defined in~\eqref{eq:def_M_BA}
and~\eqref{eq:def_M_S}.

\begin{lemma}
  \label{lem:compare_torus_trivial}
  The matrix $K$ defines a cochain map
  \begin{multline*}
    \torus\, \derived(D;\, g,\, \id_D,\, G;\, h_1,\, h_2,\, \cdots,\,
    h_n) = Y \\
    \rTo^K X = \torus\, \triv(D;\, h_1,\, h_2,\, \cdots,\, h_n) \ .
  \end{multline*}
  This map is a quasi-isomorphism. If $D$ is a bounded above complex
  of projective $R$\nbd-modules, then $K$ is a homotopy equivalence $K
  \colon Y \simeq X$.
\end{lemma}

\begin{proof}
  This follows closely the proof of
  Lemma~\ref{lem:compare_triv_analogue}, and we concentrate on the
  necessary modification. The main point is to check that $K$~defines
  a cochain map; to this end we need to verify that
  \[\big( KD^Y \big)_{B,A} - \big(D^X K \big)_{B,A} = 0\] for all $A
  \subseteq B \subseteq N$.

  We focus on $\big( D^X K \big)_{B,A} = \sum_{B \supseteq \udot S
    \supseteq A} D^X_{B,S} K_{S,A}$ first, and assume that $b-a \geq
  3$. As $X$~is the totalisation of a trivial $N$\nbd-cube,
  cf.~Example~\ref{ex:torus_of_trivial}, we know that $D^X_{B,S} = 0$
  unless $\#(B \setminus S) \leq 1$, and in case of equality
  $D^X_{B,S}$ is of the form~$\pm (h_z \tensor 1 - 1 \tensor
  x_z)$. All this means that $\big( D^X K \big)_{B,A}$ equals
  \begin{multline*}
    \big( (-1)^b d \,\cdot\, (-1)^b (-1)^{ab}
    [B:A] \cdot G \circ H_{B \setminus A} \big) \tensor 1\\
    + \sum_{z \in B \setminus A} \big( [B : B \setminus z] \cdot h_z
    \circ (-1)^{b-1} (-1)^{a(b-1)} [B \setminus z : A] \cdot G \circ
    H_{(B \setminus z) \setminus A} \big) \tensor 1\\
    - \sum_{z \in B \setminus A} \big( [B : B \setminus z] \cdot
    (-1)^{b-1} (-1)^{a(b-1)} [B \setminus z : A] \cdot G \circ H_{(B
      \setminus z) \setminus A}\big) \tensor x_z
  \end{multline*}
  so that, simplifying the sign terms and introducing a term of the
  form $GdH_{B \setminus A} - GdH_{B \setminus A}$, we obtain
  \begin{multline*}
    (-1)^{ab} \big( D^X K \big)_{B,A} = \Big( [B:A] \big( (dG+Gd)
    \circ H_{B \setminus A} - GdH_{B \setminus A} \big) \Big) \tensor 1\\
    - (-1)^{b-a} \cdot \sum_{z \in B \setminus A} \big( [B : B
    \setminus z] [B \setminus z : A] \cdot h_z \circ G \circ H_{(B
      \setminus z) \setminus A} \big) \tensor 1\\
    + (-1)^{b-a} \cdot \sum_{z \in B \setminus A} \big( [B : B
    \setminus z] [B \setminus z : A] \cdot G \circ H_{(B \setminus z)
      \setminus A} \big) \tensor x_z \ .
  \end{multline*}
  The second line of this expression reduces to $\big( [B:A] \cdot
  H_{B \setminus A}\big) \tensor 1$, as we showed in the proof of
  Lemma~\ref{lem:compare_triv_analogue}; as $dG+Gd = g - \id$ this
  means that
  \begin{multline*}
    (D^X K)_{B,A} = (-1)^{ab} [B:A] \cdot ( g H_{B \setminus
      A} - GdH_{B \setminus A}) \tensor 1 \\
    + (-1)^{ab} (-1)^{b-a} \cdot \sum_{z \in B \setminus A} \big( [B :
    B \setminus z] [B \setminus z : A] \cdot G \circ H_{(B \setminus
      z) \setminus A} \big) \tensor x_z \ .
  \end{multline*}

  \smallbreak

  Plugging in all relevant definitions, and the expression for $(D^X
  K)_{B,A}$ we just obtained, we see that the difference $(KD^Y)_{B,A}
  - (D^XK)_{B,A}$ is given by a sum of two expressions: the tensor
  product of the sum~\eqref{eq:MDY-DXM} with the identity map of $L =
  R[x_{1}^{\pm 1},\, x_{2}^{\pm 1},\, \cdots,\, x_{n}^{\pm 1}]$, and
  (up to a factor of~$-(-1)^{ab}$)
  \begin{equation*}
    \sum_{z \in B \setminus A} \big( (-1)^{b-a} [B:B
    \setminus z][B \setminus z:A] + [B:A \amalg z][A \amalg z:A] \big)
    \cdot (G \circ H_{(B \setminus z) \setminus A}) \tensor x_z \ .
  \end{equation*}
  We know from the proof of Lemma~\ref{lem:compare_triv_analogue} that
  the former is trivial; to show that the latter is trivial as well it
  is enough to verify the equality
  \begin{equation}
    \label{eq:pleasant_extended}
    [B:B\setminus z][B \setminus z:A] + (-1)^{b-a} [B:A \amalg z][A
    \amalg z:A] = 0 \ .
  \end{equation}
  The left-hand side acquires a sign independent of~$z$ when removing
  an element of~$A$ from both $A$ and~$B$, by
  Lemma~\ref{lem:incidence_independent}, so that we may assume without
  loss of generality that $A = \emptyset$ and $a=0$. But as
  $[z:\emptyset]=1$ this then is precisely the pleasant combinatorial
  exercise the reader solved at the end of the proof of
  Lemma~\ref{lem:torus_meaningful}.

  \smallbreak

  The verification for $b-a < 3$ can be done along similar lines, but
  is slightly easier due to the lower number of terms involved. We
  omit the details.
\end{proof}

We introduce yet more notation: suppose we are given, in addition to
the maps~$h_i$ above, cochain maps $\alpha \colon C \rTo D$ and $\beta
\colon D \rTo C$. We set $g = \alpha \circ \beta \colon D \rTo D$, and
let $G$ as before be a homotopy $G \colon \id_D \simeq \alpha \circ
\beta = g$ so that $dG+Gd = \alpha\beta-\id$. This gives us the data
required to define the special $N$\nbd-cube $\derived(C;\, \alpha,\,
\beta,\, G;\, h_1,\,h_2,\, \cdots,\, h_n)$. We denote its mapping
torus by~$Z$, and the corresponding differential
by~$D^Z=(D^Z_{B,A})_{A,B \subseteq N}$.

We let $J$ denote the (constant) diagonal matrix with entry $\beta
\tensor 1$; this is to be considered as a module homomorphism $Y^n
\rTo Z^n$.

\begin{lemma}[\textsc{Mather} trick for mapping tori]
  \label{lem:mather}
  The matrix~$J$ defines a cochain map
  \begin{multline*}
    \torus\, \derived(D;\, \alpha\beta,\, \id_D,\, G;\, h_1,\, h_2,\,
    \cdots,\, h_n) = Y \\
    \rTo^J Z = \torus\, \derived(C;\, \alpha,\, \beta,\, G;\, h_1,\,
    h_2,\, \cdots,\, h_n) \ .
  \end{multline*}
  If $\beta$ is a quasi-isomorphism so is~$J$. If in addition $C$
  and~$D$ are bounded above complexes of projective $R$\nbd-modules,
  then $J$ is a homotopy equivalence of $L$\nbd-module complexes $J
  \colon Y \simeq Z$.
\end{lemma}

\begin{proof}
  This is straightforward; note that maps of the form $1 \tensor ?$
  and $? \tensor 1$ commute.
\end{proof}

\part{Novikov homology and finite domination}
\label{part:novikov}

Before attacking the main topic of the paper, the relationship between
finite domination and \textsc{Novikov} cohomology, we need to digress
and introduce a few auxiliary algebraic constructions: truncated
products, multi-complexes, and totalisation of multi-complexes.

\section{Truncated products}
\label{sec:truncated_products}

The set~$\mathcal{P}$ of formal \textsc{Laurent\/} series $f =
\sum_{\mathbf{a}\in \bZ^n} r_\mathbf{a} \mathbf{x}^\mathbf{a}$, where
$r_\mathbf{a} \in R$ and
\[\mathbf{x}^\mathbf{a} = x_1^{a_1} x_2^{a_2} \ldots x_n^{a_n} \text {
  for } \mathbf{a} = (a_1,\, a_2,\, \cdots,\, a_n) \in \bZ^n \ ,\] has
an obvious module structure over the \textsc{Laurent} polynomial ring
$L = R[x_1^{\pm1},\, x_2^{\pm1},\, \cdots,\, x_n^{\pm1}]$ given by
multiplication and collecting terms. We define the {\it support\/}
$\mathrm{supp}(f)$ of a formal \textsc{Laurent\/} series~$f$ to be the
set of those $\mathbf{a}\in \bZ^n$ with $r_\mathbf{a} \neq 0$. The
\textsc{Novikov} ring $R\nov{x_1,\, x_2,\, \cdots,\, x_n}$ is defined
as
\begin{multline*}
  R\nov{x_1,\, x_2,\, \cdots,\, x_n} = \{ f \in \mathcal{P} \,|\,
  \exists k \in \bN \colon k \mathbf{1} + \mathrm{supp}(f) \subseteq
  \bN^n \}\\
  = R\powers{x_1,\, x_2,\, \cdots,\, x_n}[1/(x_1x_2\cdots x_n)]\ ;
\end{multline*}
here we write $\mathbf{1} = (1,\, 1,\, \cdots,\, 1) \in \bN^n$ and
$v+A = \{v+a \,|\, a \in A\}$.

Given a $\bZ^n$\nbd-indexed family of $R$\nbd-modules $M_\mathbf{a}$,
we write the elements of the infinite product $\prod_{\mathbf{a} \in
  \bZ^n} M_\mathbf{a}$ as formal \textsc{Laurent\/} series $g =
\sum_{\mathbf{a}\in \bZ^n} m_\mathbf{a} x^\mathbf{a}$, where
$m_\mathbf{a} \in M_\mathbf{a}$ corresponds to the factor indexed by
$\mathbf{a} \in \bZ^n$. The support of~$g$ is defined as above. The
{\it truncated product\/} is defined as
\[\prodtr_{\mathbf{a} \in \bZ^n} M_\mathbf{a} = \{ g \in
\prod_{\mathbf{a} \in \bZ^n} M_\mathbf{a} \,|\, \exists k \in \bN
\colon k \mathbf{1} + \mathrm{supp}(g) \subseteq \bN^n \} \ .\] 

For later use we formulate this in slightly more fancy terms. Given a
vector $(a_{1},\, a_{2},\, \cdots,\, a_{n}) = \mathbf{a} \in \bZ^{n}$
let us introduce the symbol
\begin{equation}
  \label{eq:def_hook}
  \hook{\mathbf{a}} = \min_{1 \leq i \leq
  n} a_{i} \ .
\end{equation}
Given an element $g \in \prod M_{\mathbf{a}}$ and an integer $k \in
\bZ$ we write $\hook{\mathrm{supp} (g)} \geq k$ if every $\mathbf{a}
\in \mathrm{supp}(g)$ satisfies $\hook{\mathbf{a}} \geq k$; we say
$\hook{\mathrm{supp}(g)} = k$ if in addition $\hook{\mathrm{supp}(g)}
\not \geq k+1$. --- With this notation $\prodtr_{\mathbf{a} \in \bZ^n}
M_\mathbf{a}$ consists of those elements $g \in \prod M_{\mathbf{a}}$
such that there exists $k \in \bZ$, depending on~$g$, with
$\hook{\mathrm{supp} (g)} \geq k$.

\medbreak

Of particular interest is the case of a {\it truncated power}, having
$M_\mathbf{a} = M$ for all $\mathbf{a} \in \bZ^n$; we reserve the
notation
\[M\nov{x_1,\, x_2,\, \cdots,\, x_n} = \prodtr_{\bZ^n} M\] for
this. The truncated power $M\nov{x_1,\, x_2,\, \cdots,\, x_n}$ comes
equipped with a natural $R\nov{x_1,\, x_2,\, \cdots,\,
  x_n}$\nbd-module structure described by multiplication of formal
\textsc{Laurent\/} series and using the scalar action of~$R$
on~$M$. In the case $M=R$ we obtain an equality of the
\textsc{Novikov} ring $R\nov{x_1,\, x_2,\, \cdots,\, x_n}$ with the
truncated power $\prodtr_{\bZ^n} R$.

\begin{lemma}
  \label{lem:fin_pres_module}
  For a finitely presented $R$\nbd-module $M$ there is a canonical
  isomorphism of $R\nov{x_1,\, x_2,\, \cdots,\, x_n}$\nbd-modules
  \[M \tensor_R R\nov{x_1,\, x_2,\, \cdots,\, x_n} \iso M\nov{x_1,\,
    x_2,\, \cdots,\, x_n} \ ,\] sending the elementary tensor $m
  \tensor \sum_{\mathbf{a} \in
    \bZ^n}r_{\mathbf{a}}\mathbf{x}^{\mathbf{a}}$ to the formal
  \textsc{Laurent} series $\sum_{\mathbf{a} \in \bZ^n} (m \cdot
  r_{\mathbf{a}}) \mathbf{x}^{\mathbf{a}} \in M\nov{x_1,\, x_2,\,
    \cdots,\, x_n}$.
\end{lemma}

\begin{proof}
  The proof is standard; details for the case $n=1$ have been
  recorded, for example, in~\cite[Lemma~2.1]{homology}. One
  establishes the result for finitely generated free $R$\nbd-modules
  first, and then passes to the general case by considering a two-step
  resolution of~$M$ by finitely generated free modules.
\end{proof}

\section{Multi-complexes and their totalisation}
\label{sec:multi_complexes}

An $(n+1)$\nbd-complex is a family of $R$\nbd-modules $E^\bullet =
(E^\mathbf{b})$, indexed by $(n+1)$\nbd-tuples
$\mathbf{b}=(b_1,b_2,\cdots,b_{n+1}) \in \bZ^{n+1}$, together with
$R$\nbd-module homomorphisms
\[d_i \colon E^{\mathbf{b}} \rTo E^{\mathbf{b}+\mathbf{e}_i} \ , \quad
1 \leq i \leq n+1 \ ,\] where $\mathbf{e}_i = (0, \cdots, 0, 1, 0,
\cdots, 0)$ is the $i$th unit vector, satisfying
\[d_i \circ d_i =0 \ , \quad \text{and} \quad d_i d_j = -d_j d_i
\text{ for } 1 \leq i < j \leq n+1 \ .\] Its {\it direct sum
  totalisation\/} is the cochain complex $\totds E^{\bullet}$ defined
by
\[(\totds E^\bullet)^k = \bigoplus_{a_1, \cdots, a_n \in \bZ}
E^{[\mathbf{a},k]} \ ,\] where $[\mathbf{a},k]= (a_1,a_2,\cdots,a_n,
k-a_1-a_2-\ldots-a_n)$, with differential $d=d_1+d_2+\ldots+d_{n+1}$.

The totalisation of a commutative $N$\nbd-diagram can be described as
the totalisation of an $(n+1)$\nbd-complex; this is a matter of
checking sign conventions with the help of
Lemma~\ref{lem:interpretation}. We record this fact:

\begin{proposition}
  \label{prop:trivial_diagram_is_totalisation}
  Let $C$ be a complex of $R$\nbd-modules, and let $f_{i}$, $1 \leq i
  \leq n$, be a collection of pairwise commuting cochain self-maps
  of~$C$.  The totalisation of $\triv(C; f_1,\, f_2,\, \cdots,\, f_n)$
  is the totalisation of an $(n+1)$\nbd-fold cochain
  complex~$E^{\bullet}$ which has the module~$C^{k}$ in degrees
  $(\epsilon_1,\, \epsilon_2,\, \cdots,\, \epsilon_n,\, k)$ where
  $\epsilon_i \in \{ 0,\,1\}$ and $k \in \bZ$. The differential in
  $\mathbf{e}_{n+1}$\nbd-direction is given by the differential of~$C$
  modified by the sign $(-1)^{\epsilon_{1} + \epsilon_{2} + \ldots +
    \epsilon_{n}}$. The differential $d_{k}$ for $1 \leq k \leq n$ is
  given by $d_{k} = (-1)^{\epsilon_{1} + \epsilon_{2} + \ldots +
    \epsilon_{k-1}} \cdot f_{k}$.\qed
\end{proposition}

\begin{corollary}
  \label{cor:triv_acyclic_tot}
  If one of the maps $f_{i}$ is a quasi-isomorphism, or if $C$~is
  acyclic, then $\tot \triv(C; f_1,\, f_2,\, \cdots,\, f_n)$ is
  acyclic.
\end{corollary}

\begin{proof}
  This is standard homological algebra of multi-complexes. One way to
  prove the claim is to use a ``partial totalisation''. Let
  $E^{\bullet}$ be the multi-complex described in
  Proposition~\ref{prop:trivial_diagram_is_totalisation}. If, for
  example, $f_{1}$ is a quasi-isomorphism, we may
  define a $2$\nbd-complex $D^{*,*}$ by setting
  \[D^{p,q} = \hskip -1 em \bigoplus_{a_{2}, a_{3}, \cdots, a_{n} \in
    \bZ} \hskip -1 em E^{p, a_{2}, a_{3}, \cdots, a_{n}, q-\sum_{i}
    a_{i}} \ ,\] equipped with horizontal differential~$d_{h}$ given
  by a direct sum of maps~$d_{1}$ and vertical differential~$d_{v}$
  induced by~$d_{2} + d_{3} + \ldots + d_{n+1}$. This $2$\nbd-complex
  is concentrated in the two columns $p=0$ and $p=1$, and has acyclic
  rows by our hypothesis on~$f_{1}$. Hence its totalisation, which is
  the same as the totalisation of~$E^{\bullet}$, is acyclic.
\end{proof}

The technique used to prove the Corollary for the case of acyclic~$C$
is of interest later on. We record a variant for later use:

\begin{corollary}
  \label{cor:tot_triv_by_2-complex}
  Let $C$ and $f_{i}$ be as in
  Proposition~\ref{prop:trivial_diagram_is_totalisation}. Then the
  $2$\nbd-complex
  \begin{displaymath}
    D^{p,q} = \bigoplus_{\substack {\udot A \subseteq N \\ \#A = p}} C^{q}
  \end{displaymath}
  with vertical differential
  \begin{displaymath}
    d^{v} = \bigoplus_{\substack {A \subseteq N \\ \#A =
        p}}(-1)^{p} d_{C} \colon D^{p,q} \rTo D^{p,q+1} 
  \end{displaymath}
  and horizontal differential given by $[A \amalg j : A] \cdot f_{j}$
  (for $j \notin A$) when considered as a map from $A$\nbd-summand to
  $(A \amalg j)$\nbd-summand, satisfies
  \begin{displaymath}
    \totds D^{\bullet} = \tot \triv(C; f_1,\, f_2,\, \cdots,\, f_n)
    \ .% \tag*{\qedsymbol}
  \end{displaymath}
  This $2$\nbd-complex~$D^{\bullet}$ is concentrated in columns $0
  \leq p \leq n$.\qed
\end{corollary}

\medbreak

Back to a general $(n+1)$\nbd-complex~$E^{\bullet}$, its {\it
  truncated product totalisation\/} $\tottr E^\bullet$ is defined
similar to~$\totds$, using truncated products in place of direct sums:
\[(\tottr E^\bullet)^k = \prodtr_{\mathbf{a} \in \bZ^{n}}
E^{[\mathbf{a},k]} \ ,\] where $[\mathbf{a},k]= (a_1,a_2,\cdots,a_n,
k-a_1-a_2-\ldots-a_n) \in \bZ^{n+1}$, and with differential again
given by $d=d_1+d_2+\ldots+d_{n+1}$.

\begin{proposition}
  \label{prop:tr_tot_acyclic}
  Suppose that the $(n+1)$\nbd-complex $E^\bullet$ is such that for
  any $a_1, a_2,\, \cdots,\, a_n \in \bZ$ the cochain complex
  $E^{a_1,a_2,\cdots,a_n,*}$ with differential~$d_{n+1}$ is exact
  (that is, $E^\bullet$ is exact in
  $\mathbf{e}_{n+1}$\nbd-direction). Then its truncated product
  totalisation $\tottr E^\bullet$ is acyclic.
\end{proposition}

\begin{proof}
  Let $c \in (\tottr E^{\bullet})^{m} = \prod_{\mathbf{a} \in \bZ^{n}}
  E^{[\mathbf{a},m]}$ be a cocycle so that $d(c)=0$. We want to
  construct an element $b = (b_{\mathbf{a}})_{\mathbf{a} \in \bZ^{n}}
  \in (\tottr E^{\bullet})^{m-1}$ with $d(b)=c$. --- We have
  $\hook{\mathrm{supp}(c)} = k$, for some $k \in \bZ$; we set
  $b_{\mathbf{a}} = 0 \in E^{[\mathbf{a}, m-1]}$ whenever
  $\hook{\mathbf{a}} < k$. These elements trivially satisfy
  $d(b)_{\mathbf{a}} = c_{\mathbf{a}}$ for $\hook{\mathbf{a}} <
  k$. (Note here that while not all components of~$b$ have been
  defined yet, the expression
  \begin{equation*}
    d(b)_{\mathbf{a}} = d_{n+1} (b_{\mathbf{a}}) + \sum_{i=1}^{n} d_{i}
    (b_{\mathbf{a} - \mathbf{e}_{i}})
  \end{equation*}
  makes sense as $\hook{\mathbf{a} - \mathbf{e}_{i}} \leq
  \hook{\mathbf{a}} < k$.)

  % Let $\mathbf{k} = (k,\, k,\, \cdots,\, k) \in \bZ^{n}$, and w
  We need to find suitable elements $b_{\mathbf{a}} \in
  E^{[\mathbf{a}, m-1]}$ for $\hook{\mathbf{a}} \geq k$. We proceed by
  induction on $\ell = |\mathbf{a}| \geq nk$, where $|\mathbf{a}| =
  \sum_{1}^{n} a_{i}$.

  So suppose $\hook{\mathbf{a}} \geq k$. For $\ell = nk$ we
  necessarily have $\mathbf{a} = (k,\, k,\, \cdots,\, k)$ and
  consequently $\hook{\mathbf{a} - \mathbf{e}_{j}} < k$. We calculate,
  using the definition of~$d$ and the equalities $d_{j} \circ d_{j} =
  0$ and $d_{i} \circ d_{j} = - d_{j} \circ d_{i}$ for $i \neq j$,
  \begin{equation}
    \label{eq:use_exactness}
    \begin{aligned}
      0 &= d(c)_{\mathbf{a}} \\
      &= d_{n+1} (c_{\mathbf{a}}) + \sum_{j=1}^{n} d_{j}
      (c_{\mathbf{a} - \mathbf{e}_{j}}) \\
      &= d_{n+1} (c_{\mathbf{a}}) + \sum_{j=1}^{n} d_{j}
      \big(d(b)_{\mathbf{a} - \mathbf{e}_{j}}\big) \\
      &= d_{n+1} (c_{\mathbf{a}}) + \sum_{j=1}^{n} d_{j} \Big( d_{n+1}
      (b_{\mathbf{a} -\mathbf{e}_{j}}) + \sum_{i=1}^{n} d_{i}
      (b_{\mathbf{a} -\mathbf{e}_{j} - \mathbf{e}_{i}}) \Big) \\
      &= d_{n+1} (c_{\mathbf{a}}) + \sum_{j=1}^{n} d_{j} d_{n+1}
      (b_{\mathbf{a} -\mathbf{e}_{j}}) \\
      &= d_{n+1} \Big( c_{\mathbf{a}} - \sum_{j=1}^{n} d_{j}
      (b_{\mathbf{a} -\mathbf{e}_{j}}) \Big)
    \end{aligned}
  \end{equation}
  so that, by our exactness hypothesis, we find an element
  $b_{\mathbf{a}} \in E^{[\mathbf{a}, m-1]}$ with
  \[d_{n+1} (b_{\mathbf{a}}) = c_{\mathbf{a}} - \sum_{j=1}^{n} d_{j}
  (b_{\mathbf{a} -\mathbf{e}_{j}}) \ .\] Then by definition of~$d$
  and the construction of~$b_{\mathbf{a}}$ we have
  \begin{equation*}
    d(b)_{\mathbf{a}} = d_{n+1} (b_{\mathbf{a}}) + \sum_{j=1}^{n} d_{j}
    (b_{\mathbf{a} - \mathbf{e}_{i}})  = c_{\mathbf{a}} \ .
  \end{equation*}

  Assume now, by induction, that for some $\ell > nk$ we have already
  constructed elements $b_{\mathbf{b}}$ for $\hook{\mathbf{b}} \geq k$
  and $|\mathbf{b}| < \ell$ satisfying $d(b)_{\mathbf{b}} =
  c_{\mathbf{b}}$. Then for any $\mathbf{a} \in \bZ^{n}$ with
  $\hook{\mathbf{a}} \geq k$ and $|\mathbf{a}| = \ell$ we notice that
  either $\hook{\mathbf{a}-\mathbf{e}_{j}} < k$ or else
  $|\mathbf{a}-\mathbf{e}_{j}| < \ell$. So using our induction
  hypothesis, the calculation~\eqref{eq:use_exactness} remains valid
  for our current~$\mathbf{a}$, which allows us to find the requisite
  element $b_{\mathbf{a}}$ in exactly the same manner as before.

  We have now defined elements $b_{\mathbf{a}} \in E^{[\mathbf{a},
    m-1]}$ for all $\mathbf{a} \in \bZ^{n}$, satisfying $d(b)=c$ by
  construction. As $c$ was arbitrary, this proves that any cocycle of
  $\tottr (E^{\bullet})$ is a coboundary.
\end{proof}

\section{From mapping tori to multi-complexes}
\label{sec:tori_as_realisations}

In this paper multi-complexes are mainly used as a tool to get
alternative representations of mapping tori. To set the stage for the
following construction, let us introduce a \textsc{Laurent} polynomial
notation for the elements of $\bZ^{n}$\nbd-indexed copowers of an
$R$\nbd-module $M$ much in the spirit of
\S\ref{sec:truncated_products}:
\begin{multline*} 
  \qquad M[x_1^{\pm 1},\, x_2^{\pm 1},\, \cdots,\, x_n^{\pm
    1}] = \bigoplus_{\mathbf{a} \in \bZ^{n}} M \\
  = \Big\{ \sum_{\mathbf{a} \in \bZ^{n}} m_{\mathbf{a}}
  \mathbf{x}^{\mathbf{a}} \,|\, m_{\mathbf{a}} = 0 \text{ for almost
    all } \mathbf{a} \Big\} \qquad
\end{multline*}
The element $m_{\mathbf{a}}$ belongs of course to the summand indexed
by~$\mathbf{a}$. We have an obvious isomorphism
\begin{equation}
  \label{eq:iso_M_Laurent}
  M \tensor_{R} R[x_1^{\pm 1},\, x_2^{\pm 1},\, \cdots,\, x_n^{\pm 1}]
  \iso M[x_1^{\pm 1},\, x_2^{\pm 1},\, \cdots,\, x_n^{\pm 1}] \ ,
\end{equation}
sending the elementary tensor $m \tensor \sum_{\mathbf{a} \in
  \bZ^n}r_{\mathbf{a}}\mathbf{x}^{\mathbf{a}}$ to the formal
\textsc{Laurent\/} polynomial $\sum_{\mathbf{a} \in \bZ^n} (m \cdot
r_{\mathbf{a}}) \mathbf{x}^{\mathbf{a}} \in M[x_1^{\pm 1},\, x_2^{\pm
  1},\, \cdots,\, x_n^{\pm 1}]$.

Now let $F$ be a special $N$\nbd-cube on the $R$\nbd-module cochain
complex~$C$, specified by the usual data of differential $d = d_C$,
cochain maps $f_i$, and (higher) homotopies~$H_S$. Let $T = \tot\,F$
denote the totalisation of~$F$. Define an $(n+1)$\nbd-multi-complex
$\pL(F)^{\bullet}$ by saying that {\it at position $\mathbf{a} \in
  \bZ^n$, the complex $\pL(F)^{\mathbf{a}, *}$ in
  $\mathbf{e}_{n+1}$\nbd-direction is the shift $T\big[\hskip 0.5pt
  |\mathbf{a}| \hskip 0.5pt \big]$ of~$T$}, where $T\big[\hskip 0.5pt
|\mathbf{a}| \hskip 0.5pt \big]^\ell = T^{\ell + |\mathbf{a}|}$, with
differential as in~$T$ re-indexed suitably.  We need to specify the
differential in $\mathbf{e}_{k}$\nbd-direction, $1 \leq k \leq
n$.\footnote{We are using $N = \{1,\, 2,\, \cdots,\, n\}$ here}
Recall that, by definition of totalisation of $N$\nbd-cubes, we have
\begin{gather*}
  \pL(F)^{\mathbf{a}, \ell} = T^{\ell + |\mathbf{a}|} = \bigoplus_{A
    \subseteq N} C^{\ell + |\mathbf{a}| - a} \\
  \noalign{\noindent and similarly}
  \pL(F)^{\mathbf{a}+ \mathbf{e}_{k}, \ell} = T^{\ell + |\mathbf{a}|+1} =
  \bigoplus_{A \subseteq N} C^{\ell + |\mathbf{a}| + 1 - a} \ .
\end{gather*}
Now the differential~$d_{k}$ in $\mathbf{e}_{k}$\nbd-direction
restricted to the $A$\nbd-summand is trivial if $k \in A$. Otherwise,
it is given by multiplication with the sign $-[A \amalg k : A]$
followed by inclusion into the $A \amalg \{k\}$\nbd-summand.

The differentials anti-commute. To show that, it is enough to compare
$d_{k} d_{\ell}$ and $d_{\ell} d_{k}$ considered as maps from
$A$\nbd-summand to $B$\nbd-summand, for $A, B \subseteq N$ (that is,
we restrict to the $A$\nbd-summand and co-restrict to the
$B$\nbd-summand). For $k,\ell \leq n$ both composites are zero, by
definition of the differentials, unless $B = A \amalg \{k,\, \ell\}$,
in which case $d_{k} d_{\ell} = -d_{\ell} d_{k}$ by the simplicial
identities~\eqref{eq:simplicial_identities}. For $\ell = n+1$ and $k
\leq n$ we need to recall the definition of the differential in the
totalisation of a special $N$\nbd-cube as the matrix $\Big( (-1)^{st}
[T:S] H_{T \setminus S} \big)_{N \supseteq T \supseteq S}$. By
definition of~$d_{k}$ we only need to consider the case $ B \ni k
\notin A$. The composition $d_{n+1} d_{k}$ gives us
\[(-1)^{b(a+1)} [B:A \amalg k] H_{B \setminus (A \amalg k)} \cdot
\big(-[A \amalg k:A]\big)\] while we have
\[\big( -[B:B \setminus k] \big) \cdot (-1)^{a(b-1)} [B \setminus k:A]
H_{(B \setminus k) \setminus A}\] for the composition $d_{k}
d_{n+1}$. To say that these have opposite sign amounts to saying that
\begin{equation*}
  [B:B\setminus k][B \setminus k:A] + (-1)^{b-a} [B:A \amalg k][A
  \amalg k:A] = 0 \ ,
\end{equation*}
after cancelling a common factor of $(-1)^{ab}$. Now for $b-a = 1$ (so
that $B = A \amalg k$) this is trivial, for $b-a=2$ this is a
reformulation of the simplicial
identities~\eqref{eq:simplicial_identities}, and for $b-a \geq 3$ this
is identity~\eqref{eq:pleasant_extended} again.

\begin{proposition}
  \label{prop:torus_by_tot}
  The isomorphism \eqref{eq:iso_M_Laurent} induces an isomorphism of
  $R[x_1^{\pm 1},\, x_2^{\pm 1},\, \cdots,\, x_n^{\pm 1}]$\nbd-module
  complexes $\totds \pL(F)^{\bullet} \iso \torus F$.
\end{proposition}

\begin{proof}
  By construction we have
  \[\big(\totds \pL(F)^\bullet\big)^k = \bigoplus_{a_1, \cdots, a_n
    \in \bZ} \pL(F)^{[\mathbf{a},k]} \ ,\] where $[\mathbf{a},k]=
  (a_1,a_2,\cdots,a_n, k-a_1-a_2-\ldots-a_n)$; but
  $\pL(F)^{[\mathbf{a},k]} = T^{k}$, by definition
    of~$\pL(F)^{\bullet}$, so that 
  \[\big(\totds \pL(F)^\bullet\big)^k = \bigoplus_{\mathbf{a} \in
    \bZ^{n}} T^{k} = T^{k}[x_1^{\pm 1},\, x_2^{\pm 1},\, \cdots,\,
  x_n^{\pm 1}] \underset{\eqref{eq:iso_M_Laurent}} \iso T^{k} \tensor_{R}
  L \ .\] Next we note, using distributivity of tensor products,
  that $T \tensor_{R} L$ is the cochain complex underlying the special
  $N$\nbd-cube used to define the mapping torus of~$F$:
  \begin{multline*}
    \big( T \tensor_{R} L \big)^{k} = T^{k} \tensor_{R} L = \Big(
    \bigoplus_{A
      \subseteq N} C^{k - a} \Big) \tensor_{R} L \\
    \iso \bigoplus_{A \subseteq N} \big( C^{k - a} \tensor_{R} L \big)
    = \bigoplus_{A \subseteq N} \big( C \tensor_{R} L \big)^{k - a}
  \end{multline*}
  It is now a matter of tedious but straightforward checking that
  under these identifications the differentials $d_{1} + d_{2} +
  \ldots + d_{n+1}$ of $\totds \pL(F)^{\bullet}$ and of $\torus F = \tot
  \bar F$, with $\bar F$~as in Definition~\ref{def:high_dim_tori},
  agree. Indeed, the action of the maps $1 \tensor x_{k}$ is encoded
  in the differential~$d_{k}$, while the effect of all other structure
  maps of~$F$ (the differential $d \tensor 1$, the maps $f_{k} \tensor
  1$, the homotopies $H_{S} \tensor 1$) is captured by the
  differential in $\mathbf{e}_{n+1}$\nbd-direction.
\end{proof}

\section{Replacing $L$-module complexes by mapping tori}
\label{sec:more}

We want to show that any cochain complex~$D$ of modules over the
\textsc{Laurent} polynomial ring $L = R[x_1^{\pm 1},\, x_2^{\pm 1},\,
\cdots,\, x_n^{\pm 1}]$ can be written as a mapping torus. We can, by
restriction of scalars, consider $D$ as a complex of~$R$-modules, as
will often be done in the sequel. In particular, we can form a new
$L$\nbd-module cochain complex $D \tensor_R L$. Note that the $n$
self-maps $x_k \tensor \id - \id \tensor x_k \colon D \tensor_R L \rTo
D \tensor_R L$ commute pairwise.

\medbreak

The mapping torus $\torus\, \triv(D;\, x_1,\, x_2,\, \cdots,\, x_n)$
is, by definition, the totalisation of the commutative $N$\nbd-cubical
diagram
\[\triv(D \tensor_R L;\, x_1 \tensor \id - \id \tensor x_1,\,
x_2 \tensor \id - \id \tensor x_2,\, \cdots,\, x_n \tensor \id - \id
\tensor x_n) \ ,\] with module in cochain level~$m$ being given by
$\bigoplus_{A \subseteq N} D^{m-a} \tensor_{R} L$.  The assignment $z
\tensor p \mapsto z \cdot p$ on the $N$\nbd-summand, and $z \tensor p
\mapsto 0$ on all other summands, defines an $L$\nbd-linear cochain
map
\[\psi \colon \torus\, \triv(D;\, x_1,\, x_2,\, \cdots,\, x_n) \rTo
\Sigma^{n} D \ .\]

\begin{lemma}
  \label{lem:resolution}
  % Let $D$ be an $L$\nbd-module cochain complex. The $L$\nbd-linear
  % map~$\psi$ from $\torus\, \triv(D;\, x_1,\, x_2,\, \cdots,\, x_n)$
  % to~$\Sigma^{n}D$ induced by $z \tensor p \mapsto z \cdot p$ for $z
  % \in D^k$ and $p \in L$ is a homotopy
  % equivalence.
  The map $\psi$ is a quasi-isomorphism. It is a homotopy equivalence
  if $D$ is a bounded above complex of projective $L$\nbd-modules.
\end{lemma}

% Recall that the mapping torus in question is the totalisation of the
% commutative $N$\nbd-cubical diagram
% \[\tot \triv(D \tensor_R L;\, x_1 \tensor \id - \id \tensor x_1,\,
% x_2 \tensor \id - \id \tensor x_2,\, \cdots,\, x_n \tensor \id - \id
% \tensor x_n)\] with module in cochain level~$m$ being given by
% $\bigoplus_{A \subseteq N} D^{m-a} \tensor_{R} L$.  In these terms,
% the map~$\psi$ acts as $z \tensor p \mapsto z \cdot p$ on the
% $N$\nbd-summand, and as the zero map otherwise.

\begin{proof}%[Proof of Lemma~\ref{lem:resolution}]
  Re-write the mapping torus $\torus\, \triv(D;\, x_1,\, x_2,\,
  \cdots,\, x_n)$ as the totalisation of the $2$\nbd-complex
  $D^{\bullet}$ associated to the commutative diagram
  \[\triv(D \tensor_R L;\, x_1 \tensor \id - \id \tensor x_1,\,
  x_2 \tensor \id - \id \tensor x_2,\, \cdots,\, x_n \tensor \id - \id
  \tensor x_n)\] according to
  Corollary~\ref{cor:tot_triv_by_2-complex}. Let $\hat D^{\bullet}$
  denote the $2$\nbd-complex which agrees with~$D^{\bullet}$
  everywhere except for the $(n+1)$st column where $\hat D^{n+1,q} =
  D^{q}$; the new differentials are given as follows:
  \begin{align*}
    d_{h} \colon & \hat D^{n,q} = D^{q} \tensor_{R} L \rTo^{\psi}
    D^{q} =
    \hat D^{n+1, q} \\
    d_{v} \colon & \hat D^{n+1,q} = D^{q} \rTo^{(-1)^{n+1}d_{D}}
    D^{q+1} = \hat D^{n+1, q+1}
  \end{align*}
  Then $\totds \hat D^{\bullet}$ is isomorphic to the mapping cone
  of~$\psi$, whence it is enough to show that $\totds \hat
  D^{\bullet}$ is acyclic. For this it is clearly sufficient to verify
  that $\hat D^{\bullet}$ has exact rows; we may in fact, without loss
  of generality, restrict attention to the $0$th row $q = 0$.

  So the claim to verify is the following: {\it Given an
    $L$\nbd-module $M = D^{0}$, the map
    \begin{displaymath}
      \psi_{M} = \psi \colon D^{*,0} \rTo M[n] \ , \quad D^{n,0} = M
      \tensor_{R} L \ni z \tensor p \mapsto z \cdot p
    \end{displaymath}
    is a quasi-isomorphism from the $0$th row of~$D^{\bullet}$ to the
    module~$M$, considered as a cochain complex concentrated in
    degree~$n$.}

  The crucial observation is that the $0$th row $D^{*,0}$
  of~$D^{\bullet}$ is nothing but the mapping torus $\torus \, \triv
  (M;\, x_1,\, x_2,\, \cdots,\, x_n)$ of a trivial $N$\nbd-cube.  We
  make use of this fact as follow.  Let $S = R[x_{1},\, x_{2},\,
  \cdots,\, x_{n}]$ be the polynomial ring, % Then $M = M \tensor_{S}
  % L$, and it will be sufficient to show that the map $\psi_{M}$ is
  % obtained from the map $\psi_{L}$ by tensoring with~$M$, as long as
  % the latter is known to be a homotopy equivalence.
  %
  % We will prove $\psi_{L}$ to be a homotopy equivalence by
  % re-writing
  % is a an $\bZ^{n}$\nbd-indexed copower of homotopy equivalences of
  % $R$\nbd-module cochain complexes. We record the details. L
  and let $F$ denote the commutative cubical diagram $\triv (S;\,
  x_1,\, x_2,\, \cdots,\, x_n)$, considered as a special $N$\nbd-cube
  of $R$-module complexes. Let $\pL(F)^{\bullet}$ be the multi-complex
  associated to~$F$ according to
  \S\ref{sec:tori_as_realisations}. Then, as we have seen in
  Proposition~\ref{prop:torus_by_tot}, $\torus (F) \iso \totds
  \pL(F)^{\bullet}$. By construction $\pL(F)^{\bullet}$ has a shifted
  copy of the cochain complex $K = \totds (F)$ in $(n+1)$-direction
  everywhere. Now $K$ is actually the \textsc{Koszul} complex of~$S$
  associated with the regular sequence $(x_{1},\, x_{2},\, \cdots,\,
  x_{n})$ so that $K$~is quasi-isomorphic, via the canonical
  projection, to $S/(x_{1},\, x_{2},\, \cdots,\, x_{n})[n] = R[n]$
  (that is, the module~$R$ considered as a complex concentrated in
  degree~$n$); since $K$ consists of free $R$\nbd-modules, this
  quasi-isomorphism is actually a homotopy equivalence of
  $R$\nbd-module complexes. Upon taking $\bZ^{n}$\nbd-indexed copowers
  we thus obtain a homotopy equivalence
  \[\epsilon \colon \torus (F) \rTo^{\simeq} R[x_{1}^{\pm 1},\,
  x_{2}^{\pm 1},\, \cdots,\, x_{n}^{\pm 1}][n] = L[n]\] of
  $L$\nbd-module complexes. Note that the cochain modules of $\torus
  (F)$ are direct sums of modules of the type $S \tensor_{R} L$. Thus
  we can form the map $M \tensor_{S} \epsilon$, tensoring source and
  target of~$\epsilon$ over~$S$ with~$M$.  The source of $M
  \tensor_{S} \epsilon$ is canonically isomorphic to the mapping torus
  of $\triv (M;\, x_1,\, x_2,\, \cdots,\, x_n)$, the target $M
  \tensor_{S} L$ is canonically isomorphic to~$M$, and $\psi_{M} = M
  \tensor \epsilon$ is a homotopy equivalence of $L$\nbd-module
  complexes as required.
  % Now we make use of the {\it left\/} $S$\nbd-module structure of $S
  % \tensor_{R} L$, which is the cochain complex (concentrated in
  % degree~$0$) used to define the mapping torus $\torus\,(F)$:
  % \begin{align*}
  %   \torus\, \triv(D;\, x_1,\, x_2,\, \cdots,\, x_n) & = \torus\,
  %   \triv(D \tensor_{S} S;\, x_1,\, x_2,\, \cdots,\, x_n) \\
  %   & \iso D \tensor_{S} \torus\, \triv(S;\, x_1,\, x_2,\, \cdots,\,
  %   x_n) \\
  %   & = D \tensor_{S} \torus\, (F) \\
  %   & \simeq D \tensor_{S} L[n] \\
  %   & \iso \Sigma^{n}D %\ .\tag*{\(\qed\)}
  % \end{align*}
  % (We have used that $M \tensor_{S} L \iso M$ for every
  % $L$\nbd-module~$M$.)
  %   % \noqed
\end{proof}

\begin{corollary}
  \label{cor:long_chain_quasi_iso}
  Suppose that we are given a complex~$D$ of~$L$\nbd-modules, an
  $R$\nbd-module complex~$C$, and mutually inverse homotopy
  equivalences of $R$\nbd-module complexes
  \[\alpha \colon C \rTo D \qquad \text{and} \qquad \beta \colon D \rTo
  C \ .\] Let $G: \id_D \simeq \alpha \circ \beta$ be a homotopy
  from~$\id_D$ to~$\alpha \circ \beta$. Then the cochain
  complex~$\Sigma^{n} D$ is quasi-isomorphic by $L$\nbd-linear maps to
  the complex
  \begin{equation*}
    \torus \derived(C;\,\alpha,\, \beta,\, G;\, x_1,\, x_2,\, \cdots,\, x_n) \ .
  \end{equation*}
  If in addition both $C$ and~$D$ are bounded above and consist of
  projective $R$\nbd-modules (restricting the $L$\nbd-module structure
  to~$R$ in the case of~$D$) the complexes $\Sigma^{n} D$ and $\torus
  \derived(C;\,\alpha,\, \beta,\, G;\, x_1,\, x_2,\, \cdots,\, x_n)$
  are actually homotopy equivalent as $L$\nbd-module complexes.
\end{corollary}

\begin{proof}
  We have a chain of quasi-isomorphisms
  \begin{align*} \Sigma^{n} D &
    \lTo[l>=3em]^\sim_{\ref{lem:resolution}} \torus \triv(D; x_1,\,
    x_2,\, \cdots,\, x_n) \\
    \noalign{\smallskip}
    & \lTo[l>=3em]^\sim_{\ref{lem:compare_torus_trivial}} \torus
    \derived (D;\, \alpha\beta,\, \id,\, G;\, x_1,\, x_2,\, \cdots,\,
    x_n) \\
    \noalign{\smallskip}
    & \rTo[l>=3em]^\sim_{\ref{lem:mather}} \torus
    \derived(C;\, \alpha,\, \beta,\, G;\, x_1,\, x_2,\, \cdots,\, x_n
    ) \ .
  \end{align*}
  The final part is automatic.
\end{proof}

\section{Finite domination implies vanishing of Novikov cohomology}
\label{sec:findom_implies_zero}

\begin{theorem}
  \label{thm:findom_triv_1st_quadrant}
  Suppose $D$ is a bounded cochain complex of projective
  $L$\nbd-modules. Suppose further that $D$ is $R$\nbd-finitely
  dominated. Then the induced cochain complex $D \tensor_L
  R\nov{x_1,\, x_2,\, \cdots,\, x_n}$ is acyclic, and thus
  contractible.
\end{theorem}

\begin{proof}
  By hypothesis there exists a bounded complex~$C$ of finitely
  generated projective $R$\nbd-modules, and mutually inverse
  $R$\nbd-linear homotopy equivalences $\alpha \colon C \rTo D$ and
  $\beta \colon D \rTo C$. By Corollary~\ref{cor:long_chain_quasi_iso}
  we have a homotopy equivalence of $L$\nbd-module complexes
  \begin{equation*}
    \Sigma^{n} D \simeq \torus \derived(C;\, \alpha,\, \beta,\, G;\, x_1,\,
    x_2,\, \cdots,\, x_n) =: P \ ,
  \end{equation*}
  where $G \colon \id_D \simeq \alpha \circ \beta$ is a chosen
  homotopy.  Consequently, the complexes remain homotopy equivalent
  after tensoring (over~$L$) with the \textsc{Novikov} ring
  $R\nov{x_1,\, x_2,\, \cdots,\, x_n}$, and the Theorem will be proved
  if we can show that the right-hand side becomes acyclic.

  To this end, we carry the ideas underpinning
  \S\ref{sec:tori_as_realisations} one step further, employing the
  notion of truncated product totalisations instead of ordinary
  totalisation. Let $F = \derived(C;\, \alpha,\, \beta,\, G;\, x_1,\,
  x_2,\, \cdots,\, x_n)$ so that $P = \torus F$, and let
  $\pL(F)^{\bullet}$ be the multi-complex associated to~$F$ as in
  \S\ref{sec:tori_as_realisations}. Then $P \iso \tot
  \pL(F)^{\bullet}$, as observed in
  Proposition~\ref{prop:torus_by_tot}. By virtually the same argument,
  we see that
  \[P \tensor_{L} R\nov{x_1,\, x_2,\, \cdots,\, x_n} \iso \tottr
  \pL(F)^{\bullet} \ ;\] this uses the cancellation rule
  \[M \tensor_{R} L \tensor_{L} R\nov{x_1,\, x_2,\, \cdots,\, x_n}
  \iso M \tensor_{R} R\nov{x_1,\, x_2,\, \cdots,\, x_n}\] and
  Lemma~\ref{lem:fin_pres_module} (and the fact that every finitely
  generated projective $R$\nbd-module is finitely presented).

  There is a chain of homotopy equivalences of $R$\nbd-module
  complexes
  \begin{align*} 0 \rTo[l>=3em]^\sim & \tot \triv(D; x_1,\, x_2,\, \cdots,\, x_n) \\
    \noalign{\smallskip}
    \lTo[l>=3em]^\sim_{\ref{lem:compare_triv_analogue}} & \tot
    \derived (D;\, \alpha\beta,\, \id,\, G;\, x_1,\, x_2,\, \cdots,\,
    x_n) \\
    \noalign{\smallskip}
    \rTo[l>=3em]^\sim_{\ref{lem:mather_for_analogue}} & \tot
    \derived(C;\, \alpha,\, \beta,\, G;\, x_1,\, x_2,\, \cdots,\, x_n
    ) = \tot(F)\ ,
  \end{align*}
  the first one from Corollary~\ref{cor:triv_acyclic_tot} (as
  multiplication by~$x_{k}$ is an isomorphism, hence a
  quasi-isomorphism). This shows that the multi-complex
  $\pL(F)^{\bullet}$ is acyclic in
  $\mathbf{e}_{n+1}$\nbd-direction. Consequently its truncated product
  totalisation is acyclic by Proposition~\ref{prop:tr_tot_acyclic}.
\end{proof}

\section{The main theorem}
\label{sec:cones}

\subsection*{Cones}

Let $M \iso \bZ^{n}$ be a lattice of rank~$n$. The group algebra
$R[M]$ is isomorphic to the \textsc{Laurent} polynomial ring in $n$
indeterminates; a choice of basis $b_{1},\, b_{2},\, \cdots,\, b_{n}$
of~$M$ determines one such isomorphism, with $\pm b_{i}$ being
identified with $x_{i}^{\pm 1}$. Abstractly, we may think of $R[M]$ as
the set of all functions $M \rTo R$ with finite support; the product
is defined via the usual convolution type formula.

Let $M_{\bR} = M \tensor_{\bZ} \bR$ be the $n$\nbd-dimensional
$\bR$-vector space associated with~$M$, and note that we have a
natural inclusion $M \subset M_{\bR}$, $m \mapsto m \tensor 1$. A {\it
  rational polyhedral cone\/} in~$M_{\bR}$, or shorter just {\it
  cone}, is a set $\sigma \subseteq M_{\bR}$ for which there exist
finitely many elements $v_{i} \in M$ with
\begin{equation*}
  \sigma = \mathrm{cone}\, \{v_{1},\, v_{2},\, \cdots,\, v_{\ell}\}
  = \{ \lambda_{1}v_{1}+ \lambda_{2}v_{2} + \ldots + \lambda_{\ell}
  v_{\ell} \,|\, \lambda_{i} \in \bR_{\geq 0}\} \ .
\end{equation*}
The {\it dimension\/} of~$\sigma$ is $\dim \sigma = \dim
\mathrm{span}_{\bR} (\sigma)$, the $\bR$\nbd-dimension of its linear
span. The {\it cospan\/} of~$\sigma$ is the largest linear subspace
contained in~$\sigma$, which is precisely $\sigma \cap (-\sigma)$. We
call $\sigma$ a {\it pointed cone\/} if its cospan is trivial (\ie,
consists of the zero vector only). Since our cones are spanned by
finitely many vectors, this is equivalent to saying that there exists
a hyperplane $H \subset M_{\bR}$ with $H \cap \sigma = \{0\}$ such
that $\sigma \setminus \{0\}$ is entirely contained in one of the open
half spaces determined by~$H$.

\begin{lemma}
  \label{lem:basis_inside_cone}
  Given an $n$\nbd-dimensional pointed cone~$\sigma$ there exists a
  basis of~$M$ consisting of elements of $M \cap \sigma$.\qed
\end{lemma}

\subsection*{Novikov rings determined by cones}

An $n$\nbd-dimensional pointed cone~$\sigma$ determines a ring $R
\nov{\sigma}$ in the following manner: Choose a vector~$w$ in the
interior of~$\sigma$ (if $\sigma$~is spanned by vectors~$v_{i}$ as in
the definition of a cone above, then $w = v_{1} + v_{2} + \ldots +
v_{\ell}$ will do), and define
\[R \nov{\sigma} = \{ f \colon M \rTo R \,|\, \exists k \geq 0: kw +
\mathrm{supp}(f) \subseteq \sigma \} \ .\footnote{We write \(v+A\) for
  the set \(\{v+a \,|\, a \in A\}\)}\] As before, multiplication is
given by a convolution type formula (which makes sense as the cone is
pointed). We call $R \nov{\sigma}$ the {\it \textsc{Novikov} ring
  associated to~$\sigma$}. It is independent from the specific choice
of vector~$w$.

\begin{example}
  In $M = \bZ^{n}$ we choose the standard basis $e_{1},\, e_{2},\,
  \cdots,\, e_{n}$ consisting of unit vectors. Let $\sigma$ be the
  cone spanned by the basis vectors. Then the identification of $R[M]$
  with the \textsc{Laurent} polynomial ring $L = R[x_1^{\pm 1},\,
  x_2^{\pm 1},\, \cdots,\, x_n^{\pm 1}]$ extends to an identification
  of $R \nov{\sigma}$ with the \textsc{Novikov} ring $R \nov{x_{1},\,
    x_{2},\, \cdots,\, x_{n}}$ introduced in
  \S\ref{sec:truncated_products}.
\end{example}

Suppose now more generally that $\sigma$ is an $n$\nbd-dimensional
cone with possibly non-trivial cospan~$U$. Write $u = \dim (U)$, and
define $\bar M_{\bR} = M_{\bR}/U$. The image~$\bar M$ of~$M$ in~$\bar
M_{\bR}$ is a lattice of rank~$n-u$ such that there is a canonical
identification $\bar M_{\bR} = \bar M \tensor_{\bZ} \bR$, and the
image~$\bar \sigma$ of~$\sigma$ in~$\bar M_{\bR}$ is a pointed
$(n-u)$-dimensional cone. Choose a vector $\bar w$ in the interior
of~$\bar \sigma$, and define
\[R \nov{\sigma} = \{\bar f \colon \bar M \rTo R[U \cap M] \,|\,
\exists k \geq 0: k \bar w + \mathrm{supp} \bar f \subseteq \bar
\sigma \} \ .\] As before, this is a ring with convolution type
product formula; in fact, $R \nov{\sigma} = R[U \cap M] \nov{\bar
  \sigma}$. We call this the {\it \textsc{Novikov} ring associated
  to~$\sigma$}. It is independent from the specific choice of
vector~$\bar w$. After choosing a splitting $M = (M \cap U) \oplus
\bar M$ we can consider $R \nov{\sigma}$ as a subset of the set
$\mathrm{map}\,(M,R)$ of maps $M \rTo R$. More explicitly, the ring
$R[U \cap M]$ is the set of maps $U \cap M \rTo R$ with finite
support; the element $\bar f \colon \bar M \rTo R[U \cap M]$
corresponds to the map
\[M = (M \cap U) \oplus \bar M \rTo R \ , \quad m = (u, \bar m)
\mapsto \bar f (\bar m) (u) \ .\]

\begin{example}
  In $M = \bZ^{n}$ we choose the standard basis $e_{1},\, e_{2},\,
  \cdots,\, e_{n}$ consisting of unit vectors. Let $\sigma$ be the
  cone spanned by the basis vectors and the vector $-e_{n}$. Then the
  identification of $R[M]$ with the \textsc{Laurent} polynomial ring
  $L = R[x_1^{\pm 1},\, x_2^{\pm 1},\, \cdots,\, x_n^{\pm 1}]$ extends
  to an identification of $R \nov{\sigma}$ with the \textsc{Novikov}
  ring $R[x_{n}^{\pm 1}] \nov{x_{1},\, x_{2},\, \cdots,\, x_{n-1}}$,
  that is, a \textsc{Novikov} ring in $(n-1)$ indeterminates over a
  \textsc{Laurent} polynomial ring in one indeterminate.
\end{example}

\subsection*{Dual cones}

The dual $N = \hom_{\bZ} (M,\, \bZ)$ of~$M$ is again a lattice of
rank~$n$; its associated vector space $N_{\bR} = N \tensor_{\bZ} \bR$
is naturally identified with the dual $\hom_{\bR} (M_{\bR},\, \bR)$
of~$M_{\bR}$. We denote by $\langle \,\cdot\,,\,\cdot\,\rangle$ the
canonical evaluation pairing $M_{\bR} \times N_{\bR} \rTo \bR$. We
define (rational polyhedral) cones in~$N_{\bR}$ just as we did
in~$M_{\bR}$ above. Given a cone $\sigma$ in~$N_{\bR}$, we define its
{\it dual\/} to be the set
\[\sigma^{\vee} = \{x \in M_{\bR} \,|\, \forall y \in \sigma:
\langle x,\, y\rangle \geq 0 \} \ .\] It can be shown that the dual of
a cone is a cone, and that
\[\dim (\sigma^{\vee}) = \mathrm{codim}\, \mathrm{cospan}\, (\sigma)\]
so that the dual of a pointed cone is an $n$\nbd-dimensional cone. We
have $\{0\}^{\vee} = M_{\bR}$; the dual of a one-dimensional pointed
cone in~$N_{\bR}$ is a closed half space in~$M_{\bR}$.

\subsection*{Fans}

A fan is a finite complex of pointed cones covering~$N_{\bR}$. More
formally, a (finite, complete) {\it fan\/} is a finite collection
$\Delta$ of pointed cones in~$N_{\bR}$ such that
\begin{enumerate}
\item the intersection of any two cones in~$\Delta$ is a face of
  either cone, and an element of~$\Delta$;
\item $\bigcup_{\sigma \in \Delta} \sigma = N_{\bR}$.
\end{enumerate}

\subsection*{The main theorem}

We are now in a position to state and prove the main result of this
paper: a complete cohomological characterisation of finite domination
using the \textsc{Novikov} rings encoded by the non-trivial cones in a
given fan.

\begin{theorem}
  \label{thm:toric_criterion}
  Let $D$ be a bounded cochain complex of finitely generated free
  $R[M]$\nbd-modules.
  \begin{enumerate}[{\rm (a)}]
  \item Suppose that $D$ is $R$\nbd-finitely dominated.
    \begin{enumerate}[{\rm (i)}]
    \item The complex $D$ is $R[M \cap U]$\nbd-finitely dominated for
      any linear subspace $U \subseteq M_{\bR}$ which is spanned by
      elements of~$M$.
    \item For every $n$\nbd-dimensional cone~$\sigma \neq M_{\bR}$ the
      induced cochain complex $D \tensor_{R[M]} R \nov{\sigma}$ is
      acyclic.
    \end{enumerate}
  \item Let $\Delta$ be a (complete) fan. Suppose that for every cone
    $\{0\} \neq \sigma \in \Delta$ the cochain complex $D \tensor_{R[M]}
    R \nov{\sigma^{\vee}}$ is acyclic. Then $D$ is $R$\nbd-finitely
    dominated.
  \end{enumerate}
\end{theorem}

\begin{proof}
  {\it Part\/}~(a): We deal with assertion~(i) first. Let $U \subseteq
  M_{\bR}$ be a subspace generated by elements of~$M$. As there is
  nothing to show if $U = \{0\}$ we assume $\dim U \geq 1$. We can
  choose a basis $b_{1},\, b_{2},\, \cdots,\, b_{u}$ of~$U$ consisting
  of elements of~$M$. %  which is part of a basis of~$M$ (that is,
  % consists of elements which are linearly independent
  % over~$\bZ$). 
  Extend this with $b_{u+1},\, b_{u+2},\, \cdots,\, b_{n}$ to a
  $\bZ$\nbd-basis of~$M$. This gives us preferred identifications of
  rings
  \begin{gather*}
    R[M \cap U] \iso R[x_{1}^{\pm 1},\, x_{2}^{\pm 1},\, \cdots,\,
    x_{u}^{\pm 1}] =: L' \\
    \noalign{\noindent and} R[M] \iso R[x_{1}^{\pm 1},\, x_{2}^{\pm
      1},\, \cdots,\, x_{u}^{\pm 1},\, x_{u+1}^{\pm 1},\, \cdots,\,
    x_{n}^{\pm 1}] = L \ .
  \end{gather*}
  Using this identification we consider $D$ as a complex of
  $L$\nbd-modules. By restriction of scalars it is also a complex
  of~$L'$\nbd-modules, consisting of free (but not necessarily
  finitely generated) modules.

  By hypothesis on~$D$ there are mutually inverse $R$\nbd-linear
  homotopy equivalences $\alpha \colon C \rTo D$ and $\beta \colon D
  \rTo C$ with $C$~a bounded cochain complex of finitely generated
  projective $R$\nbd-modules, and a homotopy $G \colon \id_{D} \simeq
  \alpha \circ \beta$. Then working with the $R$\nbd-algebra~$L'$
  instead of~$L$ we have an $L'$\nbd-linear homotopy equivalence
  \[\Sigma^{u} D \simeq \torus \derived(C;\,\alpha,\, \beta,\, G;\,
  x_1,\, x_2,\, \cdots,\, x_u) \ ,\] by
  Corollary~\ref{cor:long_chain_quasi_iso}. Now the complex on the
  right is the totalisation of a homotopy commutative cube on the
  complex $C \tensor_{R} L'$, which is bounded and consists of
  finitely generated projective $L'$\nbd-modules. This shows that
  $\Sigma^{u} D$~is $L'$\nbd-finitely dominated, hence so is~$D$.

  \smallbreak

  We now turn our attention to assertion~(ii). Let $U =
  \mathrm{cospan}\,(\sigma)$, and observe that $U$ is spanned by
  elements of~$M$ (as $\sigma$ is spanned by elements of~$M$). Write
  $\bar \sigma$ and~$\bar M$ for the images of~$\sigma$ and~$M$ in
  $\bar M_{\bR} = M_{\bR}/U$ as before. Note that $\bar M$~is a
  lattice of rank $n-u$, and $\bar \sigma$ is a cone of dimension
  $n-u$, where $u = \dim U$ as before.

  We have an isomorphism $R \nov{\sigma} \iso L' \nov{\bar
    \sigma}$. More precisely, a choice of $\bZ$\nbd-basis $b_{1},\,
  b_{2},\, \cdots,\, b_{n}$ of~$M$ as above (so that the $b_{k}$, $k
  \leq u$, form a $\bZ$\nbd-basis of~$M \cap U$) determines a
  splitting $M = (M \cap U) \oplus \bar M$; with respect to this
  splitting, $R \nov{\sigma} = L' \nov{\bar \sigma}$ as subsets of
  $\mathrm{map}\, (M,R)$.

  In fact, we can choose $b_{u+1},\, b_{u+2},\, \cdots,\, b_{n}$ in
  such a way that their respective images in~$\bar M_{\bR}$ lie
  in~$\bar \sigma$. For by Lemma~\ref{lem:basis_inside_cone} we can
  choose a $\bZ$\nbd-basis $\bar b_{u+1},\, \bar b_{u+2},\, \cdots,\,
  \bar b_{n}$ of~$\bar M$ consisting of elements of $\bar \sigma \cap
  M$, and these elements can be lifted along $M \rTo \bar M$.  With
  respect to this particular choice of basis, and with respect to the
  splitting $M = (M \cap U) \oplus \bar M$ it entails, we have
  \[L' \nov {x_{u+1},\, x_{u+2},\, \cdots,\, x_{n}} \subseteq L' \nov
  {\bar \sigma} = R \nov {\sigma}\] as subsets of
  $\mathrm{map}\,(M,R)$.

  Now as the $L$\nbd-module complex~$D$ is $L'$\nbd-finitely dominated
  by~(i) we know from Theorem~\ref{thm:findom_triv_1st_quadrant},
  applied to the ground ring~$L'$ instead of~$R$, that the complex $D
  \tensor_{L} L' \nov {x_{u+1},\, x_{u+2},\, \cdots,\, x_{n}}$ is
  acyclic and, being a bounded complex of free modules, is thus
  contractible; for this to make sense we also have to note that $L =
  L' [x_{u+1},\, x_{u+2},\, \cdots,\, x_{n}]$, thanks to our choice of
  basis elements~$b_{k}$. It follows that tensoring further over $L'
  \nov {x_{u+1},\, x_{u+2},\, \cdots,\, x_{n}}$ with $L' \nov {\bar
    \sigma} = R \nov {\sigma}$ results in a contractible, and thus
  acyclic, complex. But the result is isomorphic to the complex $D
  \tensor_{L} R \nov{\sigma}$, whence assertion~(ii) is proved.

  \medbreak

  {\it Part\/}~(b): Let $\xi \in N_{\bR}$ be an arbitrary non-zero
  vector; recall that $N$~is the dual of~$M$ so that $\xi$ determines
  a linear form on~$M_{\bR}$. The associated set
  \[R[M]_{\xi}^{\wedge} = \{f \colon M \rTo R \,|\, \forall r \in \bR:
  \# \mathrm{supp}(f) \cap \xi\inv \ignore[(-\infty,\, r\ignore)] <
  \infty\}\] carries a ring structure with multiplication given by
  convolution. By a theorem of \textsc{Sch\"utz}
  \cite[Theorem~4.7]{Schuetz-Sigma} it will be enough to show that $D
  \tensor_{R[M]} R[M]_{\xi}^{\wedge}$ is acyclic. But this is easy: as
  our fan covers all of~$N_{\bR}$ there is a (unique) smallest cone
  $\sigma \in \Delta$ with $\xi \in \sigma$.  ``Smallest'' here means
  with respect to inclusion, or equivalently with respect to
  dimension; in any case, $\sigma$ is the (unique) cone that
  contains~$\xi$ in its (relative) interior (and not in one of its
  proper faces). As $\xi \neq 0$ we have $\sigma \neq \{0\}$. Then $R
  \nov{\sigma} \subseteq R[M]_{\xi}^{\wedge}$ and
  \[D \tensor_{R[M]} R[M]_{\xi}^{\wedge} \iso D \tensor_{R[M]} R
  \nov{\sigma} \tensor_{R \nov{\sigma}} R[M]_{\xi}^{\wedge} \ .\] But
  $D \tensor_{R[M]} R \nov{\sigma}$ is acyclic by hypothesis, hence
  contractible (since $D$ is bounded and consists of free modules);
  consequently, $D \tensor_{R[M]} R[M]_{\xi}^{\wedge}$ is contractible
  and thus acyclic as well.
\end{proof}

\raggedright

\end{document}